\documentclass[10pt]{amsart}
\usepackage{latexsym,amssymb}
\input epsf 

\theoremstyle{plain}
  \newtheorem{theorem}{Theorem}[section]
  \newtheorem{proposition}[theorem]{Proposition}
  
  \newtheorem{corollary}[theorem]{Corollary}
  
\theoremstyle{definition}
  \newtheorem{definition}[theorem]{Definition}
  \newtheorem{example}[theorem]{Example}
  
  \newtheorem{question}[theorem]{Question}
 \theoremstyle{remark}
  \newtheorem{remark}[theorem]{Remark}

\newcommand\qbin[3]{\left[\begin{matrix} #1 \\ #2 \end{matrix} \right]_{#3}}

\newcommand\Hilb{\operatorname{Hilb}}
\newcommand\Hom{\operatorname{Hom}}
\newcommand\Tor{\operatorname{Tor}}
\newcommand\Ind{\operatorname{Ind}}

\newcommand\sgn{\operatorname{sgn}}
\newcommand\wt{\operatorname{wt}}
\newcommand\dist{\operatorname{d}}
\newcommand\row{\operatorname{r}}
\newcommand\val{\operatorname{v}}
\newcommand\maj{\operatorname{maj}}
\newcommand\Des{\operatorname{Des}}

\newcommand\MacSchur{\operatorname{S}}
\newcommand\MacSchurZ{\operatorname{{\bf S}}}
\newcommand\HZ{\operatorname{{\bf H}}}
\newcommand\EZ{\operatorname{{\bf E}}}
\newcommand\AZ{\operatorname{{\bf A}}}

\newcommand\Symm{\mathfrak{S}}

\newcommand\xx{{\mathbf{x}}}
\newcommand\ttt{{\mathbf{t}}}

\newcommand\ZZ{{\mathbb Z}}
\newcommand\FF{{\mathbb F}}
\newcommand\QQ{{\mathbb Q}}
\newcommand\NN{{\mathbb N}}
\newcommand\CC{{\mathbb C}}

\newcommand\Qq{\hat{\mathbb Q}}

\numberwithin{equation}{section}

\title[$(q,t)$-analogues and $GL_n(\FF_q)$]
{$(q,t)$-analogues and $GL_n(\FF_q)$}

\author{Victor Reiner}
\email{reiner@math.umn.edu}

\author{Dennis Stanton}
\email{stanton@math.umn.edu}
\address{School of Mathematics\\
University of  Minnesota\\
Minneapolis, MN 55455}

\dedicatory{To Anders Bj\"orner on his $60^{th}$ birthday.}

\thanks{Authors supported by NSF grants DMS-0601010 and DMS-0503660, respectively.}



\begin{document}

\begin{abstract}
We start with a $(q,t)$-generalization of a binomial coefficient.
It can be viewed as a polynomial in $t$ that depends upon an integer
$q$, with combinatorial interpretations when $q$ is a positive integer, and 
algebraic interpretations when $q$ is the order of a finite field.
These $(q,t)$-binomial coefficients and their interpretations generalize further in two directions, one
relating to column-strict tableaux and Macdonald's ``$7^{th}$ variation'' of
Schur functions, the other relating to permutation statistics and
Hilbert series from the invariant theory of $GL_n(\FF_q)$.
\end{abstract}

\maketitle

\section{Introduction, definition and main results}
\label{intro-section}

\subsection{Definition}

   The usual {\it $q$-binomial coefficient} may be defined by 
\begin{equation}
\label{q-bin-definition}
\qbin{n}{k}{q} := 
\frac{(q;q)_n}{(q;q)_k \cdot (q;q)_{n-k}}
\end{equation}
where $(x;q)_n:=(1-q^0 x)(1-q^1x)\cdots(1-q^{n-1}x)$.
It is a central object in combinatorics, with many 
algebraic and geometric interpretations.  We recall below 
some of these interpretations and informally explain how they generalize to our
main object of study, the {\it $(q,t)$-binomial coefficient}
\begin{equation}
\label{q-t-bin-definition}
\qbin{n}{k}{q,t}:= \frac{n!_{q,t}}{k!_{q,t} \cdot (n-k)!_{q,t^{q^k}}},
\end{equation}
where 
$
n!_{q,t}:=(1-t^{q^n-1})(1-t^{q^n-q})(1-t^{q^n-q^2}) \cdots (1-t^{q^n-q^{n-1}}).
$

If $q$ is a positive integer greater than $1$, the $(q,t)$-binomial coefficient
will be shown in Section~\ref{pascal-section} 
to be a polynomial in $t$ with nonnegative coefficients.
It is not hard to see that it specializes to the 
$q$-binomial coefficient in two limiting cases:
\begin{equation}
\label{two-limits}
\lim_{t \rightarrow 1} \qbin{n}{k}{q,t} = \qbin{n}{k}{q} \quad \text{ and } \quad
\lim_{q \rightarrow 1} \qbin{n}{k}{q,t^{\frac{1}{q-1}}} = \qbin{n}{k}{t}.
\end{equation}

\vskip.1in
\noindent
{\bf Warnings:} This $(q,t)$-binomial coefficient is {\it not} a polynomial
in $q$. The parameters $q$ and $t$ here play very different
roles, unlike the symmetric role played by the variables $q$ and $t$ in the
theory of Macdonald polynomials (see e.g. \cite[Chap. VI]{Macdonald-book}).  Also note that, 
unlike the $q$-binomial coefficient, the 
$(q,t)$-binomial coefficient is {\it not} symmetric in $k$ and $n-k$.

\subsection{A word about philosophy}
\label{philosophy-subsection}

Throughout this paper, there will be $(q,t)$-versions of various combinatorial
numbers having some meaning associated with the symmetric group $\Symm_n$.
These $(q,t)$-numbers will have two specializations to the {\it same} $q$-version or $t$-version,
as in \eqref{two-limits}.  The limit as $t \rightarrow 1$ will generally give
a $q$-version that {\it counts} some objects associated to $GL_n(\FF_q)$ when $q$ is a prime power.
The $q \rightarrow 1$ limit will generally give the {\it Hilbert series}, in the variable $t$,
for some graded vector space associated with the invariant theory
or representation theory of $\Symm_n$.  The unspecialized $(q,t)$-version will
generally be such a Hilbert series, in the variable $t$,
associated with $GL_n(\FF_q)$ when $q$ is a prime power.

One reason for our interest in such Hilbert series interpretations is that
they often give generating functions in $t$ with interesting properties.
One such property is the {\it cyclic sieving phenomenon} \cite{RSW} 
interpreting the specialization at $t$
equal to a $n^{th}$ root-of-unity for the $\Symm_n$ Hilbert series, and the
specialization at $t$ equal to a $(q^n-1)^{th}$ root-of-unity 
for the $GL_n(\FF_q)$ Hilbert series.  
%
As an example \cite[\S 9]{RSW}, the $q$-binomial coefficient in \eqref{q-bin-definition},
when $q$ is specialized to a root-of-unity of order $d$ dividing $n$, 
counts the number of $k$-element subsets of
$\{1,2,\ldots,n\}$ stable under the action of any power of the $n$-cycle $c=(1,2,\ldots,n)$
that shares the same multiplicative order $d$.  Analogously, the $(q,t)$-binomial 
coefficient in \eqref{q-t-bin-definition},
when $t$ is specialized to a root-of-unity of order $d$ 
dividing $q^n-1$, counts the number of $k$-dimensional
$\FF_q$-subspaces of $\FF_{q^n}$ stable under multiplication 
by any element of $\FF_{q^n}^\times$ that shares the same
multiplicative order $d$.

\subsection{Partitions inside a rectangle, and subspaces}

  The binomial coefficient $\binom{n}{k}$ counts $k$-element subsets of a set with $n$
elements, but also counts integer partitions $\lambda$ whose Ferrers diagram fits
inside a $k \times (n-k)$ rectangle.  The usual $q$-binomial coefficient 
$q$-counts the same set of partitions:
\begin{equation}
\label{q-bin-rectangle-interpretation}
\qbin{n}{k}{q} = \sum_{\lambda} q^{|\lambda|}
\end{equation}
where $|\lambda|$ denotes the number partitioned by $\lambda$.  
It will be shown in Section~\ref{combinatorial-interpretations-section}
that the $(q,t)$-binomial has a similar combinatorial interpretation:
\begin{equation}
\label{q-t-bin-rectangle-interpretation}
\qbin{n}{k}{q,t} = \sum_{\lambda} \wt(\lambda;q,t)
\end{equation}
where the sum runs over the same partitions $\lambda$ 
as in \eqref{q-bin-rectangle-interpretation}.  Here $\wt(\lambda; q,t)$ has
simple product expressions, showing that for integers $q \geq 2$
it is a polynomial in $t$ with nonnegative
coefficients, and that 
\begin{equation}
\lim_{t \rightarrow 1} \wt(\lambda; q,t) = q^{|\lambda|}
\quad \text{ and } \quad
\lim_{q \rightarrow 1} \wt(\lambda; q,t^{\frac{1}{q-1}}) = t^{|\lambda|}.
\end{equation}

  When $q$ is a prime power, the $q$-binomial coefficient also counts 
the $k$-dimensional $\FF_q$-subspaces $U$ of an $n$-dimensional $\FF_q$-vector space $V$.
It will be shown in Section~\ref{subspace-section} that \eqref{q-t-bin-rectangle-interpretation}
may be re-interpreted as follows:
\begin{equation}
\label{q-t-subspace-interpretation}
\qbin{n}{k}{q,t} = \sum_{U} t^{s(U)}
\end{equation}
where the summation runs over all such $k$-dimensional subspaces $U$ of $V$, and
$s(U)$ is a nonnegative integer depending upon $U$.

\subsection{Principal specialization of Schur functions}

The usual $q$-binomial coefficient has a known interpretation 
(see e.g. \cite[Exer. I.2.3]{Macdonald-book}, \cite[\S7.8]{Stanley-EC2}) as the
{\it principal specialization} of a {\it Schur function}
$s_\lambda(x_1,x_2,\ldots,x_N)$, in the special case where the partition $\lambda=(k)$
has only one part:
\begin{equation}
\label{q-bin-principal-specialization-interpretation}
\qbin{n}{k}{q} = s_{(k)}(1,q,q^2,\ldots,q^{n-k}).
\end{equation}

It will be shown in Section~\ref{schur-function-section} 
that the $(q,t)$-binomial coefficient is a special case of
a different sort of principal specialization.  In \cite{Macdonald}, Macdonald
defined several variations on Schur functions, and his $7^{th}$ variation is
a family of Schur polynomials $\MacSchur_{\lambda}(x_1,x_2,\ldots,x_n)$ 
that lie in $\FF_q[x_1,\ldots,x_n]$, and are invariant under the action of $G=GL_n(\FF_q)$.  
It turns out that the principal specializations $\MacSchur_{\lambda}(1,t,t^2,\ldots,t^{n-1})$
of his polynomials can be lifted in a natural way from $\FF_q[t]$ to $\ZZ[t]$, giving
a family of polynomials we will denote $\MacSchurZ_{\lambda}(1,t,t^2,\ldots,t^{n-1})$.
When $\lambda$ has a single part $(k)$, one has
\begin{equation}
\label{q-t-bin-principal-specialization-interpretation}
\qbin{n}{k}{q,t} = \MacSchurZ_{(k)}(1,t,t^2,\ldots,t^{n-k}).
\end{equation}

Recall also that Schur functions have a combinatorial interpretation in terms of tableaux, and
hence so do their principal specializations:
\begin{equation}
\label{Schur-principal-specialization-interpretation}
s_{\lambda}(1,q,q^2,\ldots,q^n) 
  = \sum_{T} q^{\sum_i T_i}
\end{equation}
where $T$ runs over all {\it (reverse) column-strict-tableau} $T$
of shape $\lambda$ with entries in $\{0,1,\ldots,n\}$.  
It will be shown in Section~\ref{lattice-paths-section} that 
\begin{equation}
\label{MacSchur-principal-specialization-interpretation}
\MacSchurZ_{\lambda}(1,t,t^2,\ldots,t^n) 
  = \sum_{T} \wt(T;q,t).
\end{equation}
Here $T$ runs over the same set of tableaux as in 
\eqref{Schur-principal-specialization-interpretation},
and $\wt(T; q,t)$ has a simple product expression,
showing that for integers $q \geq  2$ it is a polynomial in $t$ having nonnegative
coefficients, and that 
\begin{equation}
\lim_{t \rightarrow 1} \wt( T ; q,t) = q^{\sum_i T_i}
\quad \text{ and } \quad
\lim_{q \rightarrow 1} \wt( T ; q,t^{\frac{1}{q-1}}) = t^{\sum_i T_i}.
\end{equation}

\subsection{Hilbert series}
As mentioned above, when $q$ is a prime power, 
the $q$-binomial coefficient in \eqref{q-bin-definition} counts the 
points in the {\it Grassmannian} over the finite field $\FF_q$, that is,
the homogeneous space $G/P_k$ where $G:=GL_n(\FF_q)$ and $P_k$ is the parabolic subgroup 
stabilizing a typical $k$-dimensional $\FF_q$-subspace in $\FF_q^n$.  This is related to its alternate
interpretation as the {\it Hilbert series} for a graded ring arising in the invariant
theory of the symmetric group $W=\Symm_n$:
\begin{equation}
\label{q-bin-Hilb-interpretation}
\qbin{n}{k}{t} = \Hilb(\ZZ[\xx]^{W_k}/( \ZZ[\xx]^{W}_+ ),t). 
\end{equation}
Here $\ZZ[\xx]:=\ZZ[x_1,\ldots,x_n]$ is the polynomial algebra, with the usual
action of $W=\Symm_n$ and its {\it parabolic} or {\it Young subgroup} $W_k:=\Symm_k \times \Symm_{n-k}$,
and $( \ZZ[\xx]^{W}_+)$ denotes the ideal within the ring of $W_k$-invariants generated by
the $W$-invariants $\ZZ[\xx]^W_+$ having no constant term.  The original motivation for our
definition of the $(q,t)$-binomial came from its analogous interpretation in \cite[\S 9]{RSW} as
a Hilbert series:
\begin{equation}
\label{q-t-bin-Hilb-interpretation}
\qbin{n}{k}{q,t} = \Hilb(\FF_q[\xx]^{P_k}/( \FF_q[\xx]^{G}_+ ),t),
\end{equation}
where here the polynomial algebra $\FF_q[\xx]:=\FF_q[x_1,\ldots,x_n]$ carries the usual
action of $G=GL_n(\FF_q)$ and its parabolic subgroup $P_k$.

\subsection{Multinomial coefficients}

The above invariant theory interpretations extend naturally from binomial to {\it multinomial} coefficients.
Given an ordered composition $\alpha=(\alpha_1,\ldots,\alpha_\ell)$ of $n$ into nonnegative
parts, one has the {\it multinomial coefficient} $\binom{n}{\alpha}$ counting
cosets $W/W_\alpha$ where $W=\Symm_n$ and $W_\alpha$ is a parabolic/Young subgroup.
Alternatively, one can view $W$ as a Coxeter system with the adjacent transpositions
as generators, and this multinomial coefficient counts
the minimum-length coset representatives for $W/W_\alpha$.  These representatives
$w$ are characterized by the property that the composition $\alpha$ refines the
{\it descent composition} $\beta(w)$ that lists the lengths of the 
maximal increasing consecutive subsequences of the sequence $w=(w(1),w(2),\ldots,w(n))$.

Generalizing the multinomial coefficient is the usual 
{\it $q$-multinomial coefficient} which we recall in Section~\ref{multinomial-section}.  
It is a polynomial in $q$ with nonnegative integer
coefficients, that for prime powers $q$ counts points in the finite {\it partial flag manifold}
$G/P_\alpha$;  here $G=GL_n(\FF_q)$ and $P_\alpha$ is a parabolic subgroup.  One also has these two
interpretations, one algebraic, one combinatorial:
\begin{equation}
\label{q-multinomial-interpretations}
\begin{aligned}
 \qbin{n}{\alpha}{t} \,\, = \,\, \Hilb(\,\, \ZZ[\xx]^{W_\alpha}/( \ZZ[\xx]^{W}_+ )\,\, ,t) 
                      \,\, = \,\, \sum_{\substack{w \in W:\\ \alpha \text{ refines }\beta(w)}} t^{\ell(w)},\\
\end{aligned}
\end{equation}
where $\ell(w)$ denotes the {\it number of inversions} (or {\it Coxeter group length}) of $w$.
We will consider in Section~\ref{multinomial-section}
a {\it $(q,t)$-multinomial coefficient} with two interpretations $q$-analogous to 
\eqref{q-multinomial-interpretations}:
\begin{equation}
\label{q-t-multinomial-interpretations}
\qbin{n}{\alpha}{q,t}  \,\,=  \,\,\Hilb( \,\, \FF_q[\xx]^{P_\alpha}/( \FF_q[\xx]^{G}_+ ) \,\, ,t)
 \,\, = \,\,  \sum_{\substack{w \in W:\\ \alpha \text{ refines }\beta(w)}} \wt(w;q,t).
\end{equation}
Here $\wt(w; q,t)$ has a product expression,
showing that for integers $q \geq 2$ it is a polynomial in $t$ with nonnegative
coefficients, and that 
\begin{equation}
\lim_{t \rightarrow 1} \wt( w ; q,t) = q^{\ell(w)}
\quad \text{ and } \quad
\lim_{q \rightarrow 1} \wt( w ; q,t^{\frac{1}{q-1}}) = t^{\ell(w)}.
\end{equation}

\subsection{Homology representations and ribbon numbers}

The previous interpretations of multinomial coefficients are closely related, 
via {\it inclusion-exclusion}, to what we will 
call the {\it ribbon}, {\it $q$-ribbon}, and {\it $(q,t)$-ribbon numbers} 
for a composition $\alpha$ of $n$:
$$
\begin{aligned}
r_\alpha   &:=| \{w \in W:\alpha = \beta(w)\} |, \\
r_\alpha(q)&:= \sum_{\substack{w \in W:\\ \alpha = \beta(w)}} q^{\ell(w)},\\
r_\alpha(q,t)&:= \sum_{\substack{w \in W:\\ \alpha = \beta(w)}} \wt(w; q,t).\\
\end{aligned}
$$

It is known that the ribbon number $r_\alpha$  has an expression as a 
determinant involving factorials, going back to MacMahon.  
Stanley gave an analogous determinantal expression for the $q$-ribbon number $r_\alpha(q)$ 
involving $q$-factorials. Section~\ref{Macmahon-section} discusses an analogous 
determinantal expression for the $(q,t)$-ribbon number,
involving $(q,t)$-factorials.  

The ribbon number $r_\alpha$ can also be interpreted
as the rank of the only non-vanishing homology group in the $\alpha$-rank-selected subcomplex 
of the {\it Coxeter complex} for $W=\Symm_n$.  This homology carries an interesting and
well-studied $\ZZ W$-module structure that we will call $\chi^\alpha$.
This leads (see Theorem~\ref{hilbdet} and Remark~\ref{cell-repn-remark}) 
to the homological/algebraic interpretation
\begin{equation}
\label{q-multinomial-rep-interpretation}
r_\alpha(t) = \Hilb(\,\, M/\ZZ[\xx]^W_+M\,\, ,t)
\end{equation}
where $M$ is the graded $\ZZ[\xx]^W$-module $\Hom_{\ZZ W}( \chi^\alpha, \ZZ[\xx] )$ 
which is the $W$-intertwiner space between the homology $W$-representation $\chi^\alpha$
and the polynomial ring $\ZZ[\xx]$.

For prime powers $q$, Bj\"orner reinterpreted $r_\alpha(q)$ 
as the rank of the homology in the $\alpha$-rank-selected subcomplex 
of the {\it Tits building} for $G=GL_n(\FF_q)$.
Section~\ref{building-section}
then explains the analogous homological/algebraic interpretation:
\begin{equation}
\label{q-t-multinomial-building-interpretation}
r_\alpha(q,t) = \Hilb(\,\, M/\FF_q[\xx]^G_+M\,\, ,t)
\end{equation}
where $M$ is the graded $\FF_q[\xx]^G$-module $\Hom_{\FF_q G}( \chi_q^\alpha, \FF_q[\xx] )$ 
which is the $G$-intertwiner space between the homology $G$-representation $\chi^\alpha_q$ on
the $\alpha$-rank-selected subcomplex of the Tits building and the polynomial
ring $\FF_q[\xx]$. This last interpretation generalizes work of 
Kuhn and Mitchell \cite{KuhnMitchell}, who dealt with the case where $\alpha=1^n:=(1,1,\ldots,1)$
in order to determine the composition multiplicities of the {\it Steinberg character}
of $GL_n(\FF_q)$ within each graded component of the polynomial algebra $\FF_q[\xx]$.

\subsection{The coincidence for hooks}

In an important special case, there is a coincidence between 
the principal specializations $\MacSchurZ_{\lambda}(1,t,t^2,\ldots)$,
and the $(q,t)$-ribbon numbers  $r_\alpha(q,t)$. Specifically, 
when $\lambda=(m,1^k)$ (a {\it hook}) and $\alpha=(1^k,m)$ (the {\it reverse hook}),
we will show in Section~\ref{hook-section} that for $n \geq k$ one has
\begin{equation}
\label{hook-coincidence-formula}
\MacSchurZ_{\lambda}(1,t,t^2,\ldots,t^n) = \qbin{m+n}{n-k}{q,t} r_\alpha(q,t^{q^{n-k}}).
\end{equation}
In particular, when $k=0$, 
both sides revert to the $(q,t)$-multinomial $\qbin{m+n}{n}{q,t}$, and
when $n=k$, one has the coincidence 
$
\MacSchurZ_{\lambda}(1,t,t^2,\ldots,t^k) = r_\alpha(q,t),
$
generalizing the case $\alpha=1^k$ studied by Kuhn and Mitchell \cite{KuhnMitchell}.

\tableofcontents

\section{Making sense of the formal limits}
\label{limits-section}

Recall the definition of the $(q,t)$-multinomial from \eqref{q-t-bin-definition}:
\begin{equation}
\label{recalled-q-t-binomial-definition}
\qbin{n}{k}{q,t}:= \frac{n!_{q,t}}{k!_{q,t} \cdot (n-k)!_{q,t^{q^k}}} 
= \prod_{i=1}^{k} \frac{1-t^{q^n-q^{i-1}}}{1-t^{q^k-q^{i-1}}} ,
\end{equation}
where 
$
n!_{q,t}:=(1-t^{q^n-1})(1-t^{q^n-q})(1-t^{q^n-q^2}) \cdots (1-t^{q^n-q^{n-1}}).
$
We pause to define here carefully the ring where such generating functions
involving $q$ and $t$ live, and how the various formal limits in the Introduction will make sense.  

Let 
$$
\ttt:=(\ldots,t^{q^{-2}},t^{q^{-1}},t,t^{q^1},t^{q^2},\ldots)
$$
be a doubly-infinite sequence
of {\it algebraically independent} indeterminates, and let
$$
\Qq(\ttt):=\QQ(\ldots,t^{q^{-2}},t^{q^{-1}},t,t^{q^1},t^{q^2},\ldots)
$$
denote the field of rational functions in these indeterminates.
We emphasize that $q$ here is {\it not} an integer, and there is {\it no relation} between the
different variables $t^{q^r}$.  However, they are related by the 
{\it Frobenius operator} $\varphi$ acting invertibly via
$$
\begin{aligned}
\Qq(\ttt) & \overset{\varphi}{\longrightarrow} \Qq(\ttt) \\
  t^{q^r} & \overset{\varphi}{\longmapsto}  t^{q^{r+1}}.   
\end{aligned}
$$
We sometimes abbreviate this by saying $f(t) \overset{\varphi}{\longmapsto} f(t^q)$
and $f(t) \overset{\varphi^{-1}}{\longmapsto} f(t^{\frac{1}{q}})$.

Most generating functions considered in this paper are contained in the
subfield $\Qq(\ttt)_0$ of $\Qq(\ttt)$ generated by all quotients 
$\frac{t^{q^r}}{t^{q^s}}=: t^{q^r-q^s}$ of the indeterminates.  For example,
$$
n!_{q,t}=\left(1-\frac{t^{q^n}}{t}\right)\left(1-\frac{t^{q^n}}{t^{q^1}}\right) 
   \cdots \left(1-\frac{t^{q^n}}{t^{q^{n-1}}}\right)
$$
lies in this subfield $\Qq(\ttt)_0$, and hence so does the $(q,t)$-binomial coefficient.

Because the $t^{q^r}$ are algebraically independent, the homomorphism
$\Qq(\ttt) \rightarrow \QQ(t)$ that sends $t^{q^r} \mapsto t^r$ is well-defined,
and restricts to a homomorphism 
\begin{equation}
\label{q-limit-homomorphism}
\begin{aligned}
\Qq(\ttt)_0 & \longrightarrow \QQ(t) \\
t^{q^r-q^s} = \frac{t^{q^r}}{t^{q^s}} & \longmapsto  t^{r-s}.
\end{aligned}
\end{equation}
It is this homomorphism which makes sense of the
$q \rightarrow 1$ formal limits that appeared in the Introduction:  when we write
$
\lim_{q \rightarrow 1} \left[ F(q,t) \right]_{t \mapsto t^{\frac{1}{q-1}}}
$
for some element $F(q,t) \in \Qq(t)$, we mean by this the element in $\QQ(t)$ which is
the image of $F(q,t)$ under the homomorphism in \eqref{q-limit-homomorphism}.

To make sense of the $t \rightarrow 1$ formal limits, choose a 
positive integer $q \geq 2$.  The field of rational functions 
$\QQ(t)$ in a single variable $t$ lies at the bottom of a tower of field extensions
$$
\QQ(t) \subset \QQ(t^{q^{-1}}) \subset \QQ(t^{q^{-2}}) \subset \cdots 
$$
obtained by adjoining a $q^{th}$ root at each stage.
The union $\bigcup_{r \geq 0} \QQ(t^{q^{-r}})$ is a ring with a
specialization homomorphism from the ring $\Qq(\ttt)$ defined above:
\begin{equation}
\label{q-integer-homomorphism}
\begin{aligned}
\Qq(\ttt) & \longrightarrow \bigcup_{r \geq 0} \QQ(t^{q^{-r}})\\
t^{q^r} & \longmapsto t^{q^r}.
\end{aligned}
\end{equation}
Note that in \eqref{q-integer-homomorphism}, the symbol ``$t^{q^r}$" has two different meanings: 
on the left it is one of the doubly-indexed family
of indeterminates, and on the right it is the $(q^r)^{th}$ power of the variable $t$.
Many of our results will assert that various generating functions $F(q,t)$
in $\Qq(\ttt)$, when specialized as in  \eqref{q-integer-homomorphism}, have image lying in the subring
$\ZZ[t] \subset \bigcup_{r \geq 0} \QQ(t^{q^{-r}})$;  this is what is meant when we
say $F(q,t)$ in $\Qq(\ttt)$ ``is a polynomial in $t$ for integers $q \geq 2$''.
In this situation, we will write $\lim_{t \rightarrow 1} F(q,t)$ for the
result after applying the further evaluation homomorphism $\ZZ[t] \rightarrow \ZZ$
that sends $t$ to $1$.

In Section~\ref{schur-function-section}, we shall be interested in working with
$n$ variables $x_1,\ldots,x_n$, rather than just the single variable $t$.
To this end, define $\Qq(\xx)$ to be the $n$-fold tensor product
$\Qq(\ttt) \otimes \cdots \otimes \Qq(\ttt)$, renaming the variable
$t$ as $x_i$ in the $n^{th}$ tensor factor.
Here the Frobenius automorphism $\varphi$ acts by
$(\varphi f)(x_i)=f(x_1^q,\ldots,x_n^q)$.  There are various
specialization homomorphisms from $\Qq(\xx) \rightarrow \Qq(\ttt)$,
but the one that will be of interest here is the
{\it principal specialization} $\Qq(\xx) \rightarrow \Qq(\ttt)$ sending
$x_i^{q^r} \longmapsto (t^{q^r})^{i-1}$, and abbreviated by
$f(x_1,\ldots,x_n) \longmapsto f(1,t,t^2,\ldots,t^{n-1}).$

\section{Why call it a ``binomial coefficient''?}
\label{why-binomial-section}
We give two reasons for the name ``$(q,t)$-binomial coefficient''.

\subsection{Binomial theorem}


The elementary symmetric function $e_r(x_1,\ldots,x_n)$ can be
defined by the identity 
$
\prod_{i=1}^n (y + x_i)
= \sum_{s=0}^n y^{s} e_{n-s}(\xx)
$
in $\ZZ[y,x_1,\ldots,x_n]$.
When specialized to $x_i = t^{i-1}$ this gives the following
version of the {\it $t$-binomial theorem}:
\begin{equation}
\label{t-binomial-theorem}
\prod_{i=1}^n (y + t^{i-1})
= \sum_{s=0}^n y^{s} e_{n-s}(1,t,t^2,\ldots,t^{n-1})
= \sum_{s=0}^n y^{s} \qbin{n}{s}{t} t^{\binom{n-s}{2}}.
\end{equation}
On the other hand, a special case of Macdonald's 7th variation on Schur
functions, to be discussed in Section~\ref{schur-function-section},
are polynomials $E_r(x_1,\ldots,x_n)$ 
which can be defined by the following identity in
in $\FF_q[x_1,\ldots,x_n, y]$
(see \cite[Chap. I, \S2, Exer. 26, 27]{Macdonald-book}, \cite[\S 7]{Macdonald}):
\begin{equation}
\label{Dickson-definition}
\prod_{\ell(\xx) \in (\FF_q^n)^*} (y + \ell(\xx))
= \sum_{s=0}^n y^{q^s} E_{n-s}(\xx).
\end{equation}
Here the product runs over all $\FF_q$-linear functionals
$\ell(x_1,\ldots,x_n)$ on $\FF_q^n$.   We will later prove
a formula \eqref{qtelem} for the specialization
$x_i=t^{i-1}$ in $E_r(\xx)$, from which \eqref{Dickson-definition}
gives the following identity valid in $\FF_q[t,y]$:
\begin{equation}
\label{less-than-satisfactory-binomial}
\prod_{\ell(\xx) \in (\FF_q^n)^*} (y+\ell(1,t,t^2,\ldots,t^{n-1}))
= \sum_{s=0}^n y^{q^s} \qbin{n}{s}{q,t}
     \prod_{j=1}^{n-s} \frac{t^{q^n}-t^{q^{n-j}}}{t^{q^{s+j}}-t^{q^s}}.
\end{equation}
Although \eqref{less-than-satisfactory-binomial} does not have a 
rigorous limit as $q$ approaches $1$, it can viewed as a $q$-analogue of
the $t$-binomial theorem \eqref{t-binomial-theorem}.

\subsection{Binomial convolution}
Consider three algebras of generating functions.  The first
two are the power series rings $\QQ[[y]]$, $\QQ(q)[[y]]$ in a single variable 
$y$ with coefficients in the field $\QQ$ or in the rational function field 
$\QQ(q)$.  The third is an associative but noncommutative
algebra, isomorphic as $\Qq(\ttt)$-vector space to $\Qq(\ttt)[[y]]$
so that it has a $\Qq(\ttt)$-basis $\{1,y,y^2,\ldots\}$, but with its multiplication 
twisted\footnote{This is the {\it twisted semigroup algebra}
for the mulitplicative semigroup 
$\{1,y,y^2,\ldots\}$, in which the semigroup acts on the coefficients $\Qq(\ttt)$ by 
letting the generator $y$ act as $\varphi$.} as follows:
$$
f(t) y^k \cdot g(t) y^\ell := f(t) \,\, g(t^{q^k}) \,\, y^{k+\ell}.
$$
Each of these three algebras, has a {\it divided power} basis 
$\{ y^{(n)} \}_{n \geq 0}$ as a vector space over
$\QQ$ (resp. $\QQ(q), \Qq(\ttt)$),
defined by
$$
y^{(n)}  :=\frac{y^n}{n!} 
\left(  \text{ resp. }
\frac{y^n}{(q;q)_n}, \quad
\frac{y^n}{n!_{q,t}}
\right).
$$
One can readily check that the various binomials give
the structure constants for multiplication in this basis:
$$
y^{(k)} y^{(\ell)} = \binom{k+\ell}{k} y^{(k+\ell)}
\left( \text{ resp. } 
\qbin{k+\ell}{k}{q} y^{(k+\ell)}, \quad
\qbin{k+\ell}{k}{q,t} y^{(k+\ell)}
\right).
$$
Therefore one has a {\it binomial convolution formula}:
the product $A(y) B(y)$ of two exponential generating functions
$
A(y):=\sum_{k \geq 0} a_k \,\, y^{(k)},
B(y):=\sum_{\ell \geq 0} b_\ell \,\, y^{(\ell)}
$
has its coefficient of $y^{(n)}$ given by
$$
\sum_{k+\ell=n}  \binom{n}{k} a_k \,\, b_\ell
\left( \text{ resp. } 
\sum_{k+\ell=n} \qbin{n}{k}{q} a_k(q) \,\, b_\ell(q), \,\,\,\,
\sum_{k+\ell=n} \qbin{n}{k}{q,t} a_k(t) \,\,b_\ell(t^{q^k})
\right).
$$

\section{Pascal relations}
\label{pascal-section}

Our starting point will be the $(q,t)$-analogue of the two $q$-Pascal relations
for the $q$-binomial; see e.g. \cite[Prop. 6.1]{KacCheung},\cite[\S1.3]{Stanley-EC1}:
\begin{equation}
\label{q-Pascal-relations}
\begin{aligned}
\qbin{n}{k}{q}
&=          &\qbin{n-1}{k-1}{q}\,\,+ &q^k      & \qbin{n-1}{k}{q}\\
\qbin{n}{k}{q}
&=q^{n-k}   &\qbin{n-1}{k-1}{q}+  &         &\qbin{n-1}{k}{q}.
\end{aligned}
\end{equation}

\begin{proposition} 
\label{qtpascal}
If $0\le k\le n$,
\begin{equation}
\label{q-t-Pascal-relations}
\begin{aligned}
\qbin{n}{k}{q,t}
&=          &\qbin{n-1}{k-1}{q,t^q}\,\,+ &t^{q^k-1} &\frac{k!_{q,t^q}}{k!_{q,t}} \qbin{n-1}{k}{q,t^q}\\
\qbin{n}{k}{q,t}
&=t^{q^n-q^k} &\qbin{n-1}{k-1}{q,t^q}+ &        &\frac{k!_{q,t^q}}{k!_{q,t}} \qbin{n-1}{k}{q,t^q}.
\end{aligned}
\end{equation}
\end{proposition}
\begin{proof}
Both relations are straightforward to check;  we check here only the second.
We will make frequent use of the fact that 
\begin{equation}
n!_{q,t}= (1-t^{q^n-1}) \cdot (n-1)!_{q,t^q}.
\end{equation}

Starting with the right side of the second relation in \eqref{q-t-Pascal-relations}, one checks
$$
\begin{aligned}
& t^{q^n-q^k} \qbin{n-1}{k-1}{q,t^q}+ \frac{k!_{q,t^q}}{k!_{q,t}} \qbin{n-1}{k}{q,t^q} \\
&=t^{q^n-q^k} \frac{(n-1)!_{q,t^q}}{(k-1)!_{q,t^q}(n-k)!_{q,t^{q^k}}}
                  \quad + \quad \frac{k!_{q,t^q}}{k!_{q,t}}  \cdot 
                      \frac{(n-1)!_{q,t^q}}{k!_{q,t^q}(n-k-1)!_{q,t^{q^{k+1}}}}\\
&=\frac{(n-1)!_{q,t^q}}{(k-1)!_{q,t^q}(n-k-1)!_{q,t^{q^{k+1}}}}
     \left ( \frac{t^{q^n-q^k}}{1-t^{q^n-q^k}}   + \frac{1}{1-t^{q^k-1}} \right)\\
&=\frac{(n-1)!_{q,t^q}}{(k-1)!_{q,t^q}(n-k-1)!_{q,t^{q^{k+1}}}}
     \cdot \frac{1-t^{q^n-1}}{(1-t^{q^n-q^k})(1-t^{q^k-1})} 
=\qbin{n}{k}{q,t}
\end{aligned}
$$
\end{proof}

Note that the quotient $\frac{k!_{q,t^q}}{k!_{q,t}}$ appearing in both $(q,t)$-Pascal relations  
\eqref{q-t-Pascal-relations} factors as follows:
\begin{equation}
\label{factorial-quotient}
\frac{k!_{q,t^q}}{k!_{q,t}} = [q]_{t^{q^k-1}} [q]_{t^{q^k-q}}\cdots [q]_{t^{q^k-q^{k-1}}}
\end{equation}
where $[N]_t:=\frac{1-t^N}{1-t}=1+t+t^2+\cdots+t^{N-1}$.
Consequently, if $q$ is a positive integer, 
this quotient is a polynomial in $t$ with nonnegative coefficients.

\begin{corollary}
\label{qtbinpos}
If $q\ge 2$ is an integer, then $\qbin{n}{k}{q,t}$ is a polynomial in $t$ with 
nonnegative coefficients.
\end{corollary}
\begin{proof} This follows by induction on $n$ 
from either of the $(q,t)$-Pascal relations, using \eqref{factorial-quotient}.
\end{proof}

\begin{remark}
Note that taking the formal limit as $t \rightarrow 1$
in {\it either} of the $(q,t)$-Pascal 
relations \eqref{q-t-Pascal-relations} leads to the {\it same} $q$-Pascal relation,
namely the first one in \eqref{q-Pascal-relations}. 
On the other hand, replacing $t$ with $t^{\frac{1}{q-1}}$ and
then taking the other formal limit as $q \rightarrow 1$ 
in the two different $(q,t)$-Pascal relations \eqref{q-t-Pascal-relations} leads to the
{\it two different} $q$-Pascal relations in \eqref{q-Pascal-relations}.
\end{remark}

\section{Combinatorial interpretations}
\label{combinatorial-interpretations-section}

Our goal here is to explain the combinatorial interpretation
for the $(q,t)$-binomial described in \eqref{q-t-bin-rectangle-interpretation}:
$$
\qbin{n}{k}{q,t} = \sum_{\lambda} \wt(\lambda;q,t)
$$
where the sum runs over partitions $\lambda$ whose Ferrers diagram fits inside
a $k \times (n-k)$ rectangle, that is, $\lambda$ has at most $k$ parts,
each of size at most $n-k$.

\subsection{A product formula for $\wt(\lambda;q,t)$}

The weight $\wt(\lambda; q,t)$ actually depends upon the number of rows $k$ of the
rectangle in which $\lambda$ is confined, or alternatively, upon the number of
parts of $\lambda$ if one counts parts of size $0$.  To emphasize this dependence
upon $k$, we redefine the notation $\wt(\lambda;q,t):=\wt(\lambda,k;q,t)$.
This weight will be defined as a product over the cells $x$ in the Ferrers diagram for $\lambda$
of a contribution $\wt(x,\lambda, k;q,t)$ for each cell.  To this end, first define an exponent 
$$
e_k(x):=q^{\row(x)+\dist(x)}-q^{\dist(x)}
$$ 
where $\row(x)$ is the index of the row containing $x$, and 
$\dist(x)$ is the {\it taxicab} (or {\it Manhattan}, or $L^1$-) distance
from $x$ to the bottom left cell of the rectangle, that is, the cell
in the $k^{th}$ row and first column. Alternatively, 
$\dist(x)={\mathrm{content}}(x)+k-1$, 
where the content of a cell $x$ is $j-i$ if $x$ lies in row $i$ and column $j$.
Then
\begin{equation}
\label{lambda-weight-definition}
\begin{aligned}
\wt(x,\lambda,k;q,t) &:=
\begin{cases}
t^{e_k(x)} [q]_{t^{e_k(x)}} 
    & \text{ if }x\text{ is the bottom cell of }\lambda\text{ in its column}.\\
[q]_{t^{e_k(x)}}
    & \text{ otherwise.}\\
\end{cases}\\
\wt(\lambda,k; q,t) &:= 
 \prod_{x \in \lambda} \wt(x,\lambda, k;q,t).
\end{aligned}
\end{equation}

\begin{example}
\label{partition-weight-example}
Let $n=10$ and $k=4$, so $n-k=6$.  Let $\lambda=(4,3,1,0)$ inside the $4 \times 6$ rectangle.
The cells $x$ of the Ferrers diagram for $\lambda$ are labelled
with the value $d(x)$ below
$$
\begin{array}{rrrrrr}
3     & 4     & 5     & 6     & \cdot & \cdot \\
2     & 3     & 4     & \cdot & \cdot & \cdot\\
1     & \cdot & \cdot & \cdot & \cdot & \cdot\\
\cdot & \cdot & \cdot & \cdot & \cdot & \cdot\\
\end{array}
$$
and $\wt(\lambda,k;q,t)$ is the product of the corresponding factors 
$\wt(x,\lambda,k;q,t)$ shown below:
$$
\begin{array}{rrrrrrr}
&[q]_{t^{q^4-q^3}}             & [q]_{t^{q^5-q^4}}              & [q]_{t^{q^6-q^5}}          & t^{q^7-q^6} [q]_{t^{q^7-q^6}}  \\
&[q]_{t^{q^4-q^2}}             &  t^{q^5-q^3} [q]_{t^{q^5-q^3}} & t^{q^6-q^4} [q]_{t^{q^6-q^4}} & \cdot \\
&t^{q^4-q^1} [q]_{t^{q^4-q^1}} & \cdot                          & \cdot & \cdot \\
&\cdot & \cdot & \cdot & \cdot 
\end{array}
$$
\end{example}
One can alternatively define $\wt(\lambda, k; q,t)$ recursively.
Say that $\lambda$ fitting inside a $k \times (n-k)$ rectangle
has a {\it full first column} if $\lambda$ has $k$ nonzero parts; in this case,
let $\hat\lambda$ denote the partition inside a $k \times (n-k-1)$ rectangle
obtained by removing this full first column from $\lambda$.  

\begin{proposition}
\label{lambda-weight-recurrence}
One can characterize $\wt(\lambda, k; q,t)$ by the recurrence
$$
\wt(\lambda, k; q,t) =
\begin{cases}
t^{q^k-1}\frac{k!_{q,t^q}}{k!_{q,t}} 
  \wt(\hat\lambda, k; q,t^q) & \text{ if }\lambda\text{ has a full first column,}\\
\wt(\lambda, k-1; q,t^q)   & \text{ otherwise,}
\end{cases}
$$ 
together with the initial condition $\wt(\varnothing,k;q,t):=1.$
\end{proposition}
\begin{proof}
This is straightforward from the definition \eqref{lambda-weight-definition} and
equation \eqref{factorial-quotient}.
\end{proof}

\begin{theorem}
\label{box-interpretation-theorem}
With the $(q,t)$-binomial defined as in \eqref{q-t-bin-definition}
and $\wt(\lambda,k; q,t)$ defined as above, one has
$$
\qbin{n}{k}{q,t} = \sum_{\lambda} \wt(\lambda , k;q,t)
$$
where the sum ranges over all partitions $\lambda$ whose Ferrers diagram fits inside
a $k \times (n-k)$ rectangle.

Furthermore,
\begin{equation}
\lim_{t \rightarrow 1} \wt(\lambda, k; q,t) = q^{|\lambda|}
\quad \text{ and } \quad
\lim_{q \rightarrow 1} \wt(\lambda, k; q,t^{\frac{1}{q-1}}) = t^{|\lambda|}.
\end{equation}
\end{theorem}

\noindent
Note that this gives a second proof of Corollary~\ref{qtbinpos}.

\begin{proof}
The first assertion of the theorem then follows by induction on $n$, comparing
the first $(q,t)$-Pascal relation in \eqref{q-t-Pascal-relations} with
Proposition~\ref{lambda-weight-recurrence};  classify
the terms in the sum $\sum_{\lambda} \wt(\lambda , k;q,t)$ according to whether
$\lambda$ does not or does have a full first column.

The limit evaluations in the theorem also follow by induction on $n$, using
the following basic limits:
\begin{equation}
\label{basic-limits}
\begin{aligned}
\lim_{t \rightarrow 1} [q]_{t^{q^{r+d}-q^d}} = q, 
& \qquad
\lim_{t \rightarrow 1} t^{q^{r+d}-q^d} = 1\\
\lim_{q \rightarrow 1} \left( [q]_{t^{q^{r+d}-q^d}}  \right)_{t \mapsto t^{\frac{1}{q-1}}} = 1,
& \qquad
\lim_{q \rightarrow 1}  \left( t^{q^{r+d}-q^d}\right)_{t \mapsto t^{\frac{1}{q-1}}} = t^r,
\end{aligned}
\end{equation}
and hence
\begin{equation}
\begin{aligned}
\lim_{t \rightarrow 1} \frac{k!_{q,t^q}}{k!_{q,t}} &= q^k,\\
\lim_{q \rightarrow 1} 
   \left( \frac{k!_{q,t^q}}{k!_{q,t}} \right)_{t \mapsto t^{\frac{1}{q-1}}} &= t^k.
\end{aligned}
\end{equation}
\end{proof}

In the remainder of this section, we reformulate Theorem~\ref{box-interpretation-theorem} 
in two other ways, each with their own advantages.

\subsection{A different partition interpretation}

First, note that \eqref{recalled-q-t-binomial-definition} implies that
for fixed $k \geq 0$ one has 
$$
\lim_{n \rightarrow \infty} \qbin{n}{k}{q,t}
=\prod_{i=1}^{k} \frac{1}{1-t^{q^{k}-q^{k-i}}}.
$$
For integers $q \geq 2$, this is the generating function in $t$ counting
integer partitions $\mu$ whose part sizes are restricted to the set
$\{ q^k-1, q^k-q, \ldots, q^k-q^{k-1} \}$.
One might ask whether there is an analogous result for the case where $n$ is finite, perhaps
by imposing some restriction on the multiplicities of the above parts in $\mu$.  This turns
out indeed to be the case, as we now explain.

\begin{definition}
Given $\lambda$ a partition with at most $k$ nonzero parts, set $\lambda_{k+1}=0$, and
define for $i=1,2,\ldots,k$
$$
\delta_i(\lambda):= \,\, \sum_{j=\lambda_{i+1}}^{\lambda_i-1} q^j 
=q^{\lambda_{i+1}}[\lambda_i-\lambda_{i+1}]_q.
$$
For an integer $q \geq 2$, say that a partition $\mu$
is {\it $q$-compatible with $\lambda$} if  the parts of $\mu$ are restricted to
the set $\{q^k-1, q^k-q, \ldots, q^k-q^{k-1}\}$, and for each $i=1,2,\ldots,k$, the
multiplicity of the part $q^k-q^{k-i}$ lies in the semi-open interval 
$[\delta_i(\lambda),\delta_i(\lambda)+q^{\lambda_i})$.
\end{definition}

For example, if $n=5$, $k=2$ and $\lambda=31$, the partitions $\mu$ 
which are $q$-compatible with $\lambda$ are of the form $\mu=(q^2-q)^{m_1} (q^2-1)^{m_2}$
where
$$
q^1+q^2\le m_1\le q^1+q^2+q^3-1, \quad
1\le m_2\le q.
$$

It turns out that $\mu$ determines $\lambda$ in this situation:
\begin{proposition}
\label{compatibility-uniqueness}
Given a partition $\mu$ with parts restricted to the set $\{q^k-q^{k-i}\}_{i=1}^k$,
there is at most one partition $\lambda$ having $k$ nonzero parts
or less which can be $q$-compatible with $\mu$.
\end{proposition}
\begin{proof}
Recall that $\lambda_{k+1}=0$.  Then for 
$i=k,k-1,\ldots,2,1$, check that the multiplicity $m_i$ of the part $q^k-q^{k-i}$
in $\mu$ determines $\lambda_i$ uniquely, by downward induction on $i$.
\end{proof}

\begin{theorem} 
\label{compat} 
\begin{equation}
\label{collated-rows-formula}
\qbin{n}{k}{q,t} = 
  \sum_\lambda \qquad \prod_{i=1}^k (t^{q^k-q^{k-i}})^{\delta_i(\lambda)} \,\, [q^{\lambda_i}]_{t^{q^k-q^{k-i}}}
\end{equation}
where the sum ranges over all partitions $\lambda$ fitting inside a $k \times (n-k)$ rectangle.

Consequently, for integers $q \geq 2$,
\begin{equation}
\label{compatible-partitions-formula}
\qbin{n}{k}{q,t} =
  \sum_\mu t^{|\mu|}
\end{equation}
where the sum runs over partitions $\mu$ which are $q$-compatible with a $\lambda$ that fits
inside a $k \times (n-k)$ rectangle.
\end{theorem}

\begin{proof}
Equation \eqref{compatible-partitions-formula} follows immediately from
equation \eqref{collated-rows-formula}.  The latter follows from Theorem~\ref{box-interpretation-theorem}
if one can show that
$$
\wt(\lambda,k;q,t)=
\prod_{i=1}^k (t^{q^k-q^{k-i}})^{\delta_i(\lambda)} \,\, [q^{\lambda_i}]_{t^{q^k-q^{k-i}}}.
$$
This would follow from showing that for each row $i=1,2,\ldots,k$ in $\lambda$,
one has 
\begin{equation}
\label{row-telescoping}
\prod_{\substack{x \in \lambda:\\ \row(x)=i}} \wt(x,\lambda,k;q,t) 
  = (t^{q^k-q^{k-i}})^{\delta_i(\lambda)} \,\, [q^{\lambda_i}]_{t^{q^k-q^{k-i}}}.
\end{equation}
This is not hard.  The left side of \eqref{row-telescoping} is easily checked to equal
$$
(t^{q^k-q^{k-i}})^{\delta_i(\lambda)} [q]_{t^{q^0(q^k-q^{k-i})}} [q]_{t^{q^1(q^k-q^{k-i})}} 
  \cdots [q]_{t^{q^{\lambda_i-1}(q^k-q^{k-i})}}.
$$
Repeated apply to this expression the identity
$$
[M]_Q \cdot [N]_{Q^M} = [MN]_{Q},
$$
with $Q=t^{q^k-q^{k-i}}$ and $N=q$ each time, and the result is the right side of
\eqref{row-telescoping}.
\end{proof}

Note that by setting $t=1$ in \eqref{compatible-partitions-formula},
one obtains a new interpretation for the usual $q$-binomial coefficient,
as the cardinality of a set of integer partitions, rather than as a generating function in $q$:

\begin{corollary} For integers $q \geq 2$, the
the integer $\qbin{n}{k}{q}$ is the number of partitions 
$q$-compatible with a $\lambda$ that fits inside a $k \times (n-k)$ rectangle.
\end{corollary}

%
%
%
%
%

\subsection{A subspace interpretation}
\label{subspace-section}

The $q$-binomial coefficient is the number of $k$-spaces of an 
$n$-space over a finite field with $q$ elements; 
see e.g. \cite[\S 7]{KacCheung}, \cite[Prop. 1.3.18]{Stanley-EC1}. 
We give in Theorem \ref{qtsubspace} 
a statistic on these subspaces whose generating function in $t$ is the $(q,t)$-binomial coefficient.

Fix a basis for the $n$-dimensional space $V$ over $\mathbb{F}_q$.
Relative to this basis, any $k$-dimensional subspace $U$ of $V$
is the row-space of a unique matrix $k \times n$ matrix $A$
over $\FF_q$ in {\it row-reduced echelon form}.  This matrix $A$ will have
exactly $k$ pivot columns, and the sparsity pattern for the (possibly) nonzero entries in its nonpivot columns
have an obvious bijection to the cells $x$ in a partition $\lambda$ inside a $k \times (n-k)$
rectangle; call these $|\lambda|$ entries $a_{ij}$ of $A$ the {\it parametrization} entries for $U$.

\begin{example}
If $n=10$ and $k=4$ one possible such row-reduced echelon form could be
$$
U=\left[
\begin{matrix}
0 & * &  0  & * & * & 0 & * & 1 & 0 & 0 \\
0 & * &  0  & * & * & 1 & 0 & 0 & 0 & 0 \\
0 & * &  1  & 0 & 0 & 0 & 0 & 0 & 0 & 0 \\
1 & 0 &  0  & 0 & 0 & 0 & 0 & 0 & 0 & 0 
\end{matrix}
\right]
$$     
where each $*$ is a parametrization entry representing an element of the field $\FF_q$.
The associated partition $\lambda=(4,3,1,0)$ fits inside a $4 \times 6$ rectangle, and 
is the one considered in Example~\ref{partition-weight-example} above:
$$
\begin{matrix}
*     &  * & * & * & \cdot & \cdot  \\
*     &  * & * & \cdot  & \cdot  & \cdot    \\
*     &  \cdot  & \cdot  & \cdot  & \cdot  & \cdot   \\
\cdot      &  \cdot  & \cdot  & \cdot  & \cdot  & \cdot  
\end{matrix}
$$ 
\end{example}

To define the statistic $s(U)$, first fix two bijections $\phi_0, \phi_1$:
$$
\begin{aligned}
\FF_q &\overset{\phi_0}{\rightarrow}\{0,1,\ldots,q-1\}\\
\FF_q &\overset{\phi_1}{\rightarrow}\{1,2,\dots,q\}.
\end{aligned}
$$
For each parametrization entry $a_{ij}$ of $A$,
define the integer value $\val_U(a_{ij})$ to be $\phi_1(a_{ij})$
or $\phi_0(a_{ij})$, depending on whether or not $(i,j)$ is the
lowest parametrization entry in its column of $A$.
Then define
$$
\begin{aligned}
\dist_U(i,j) &:=k-i+j - |\{\text{pivot columns left of }j\}|-1\\
s(U)&:= \sum_{(i,j)} \val_U(a_{ij}) (q^{i+\dist_U(i,j)}-q^{\dist_U(i,j)}) 
\end{aligned}
$$
where the summation has $(i,j)$ ranging over all parametrization positions in the row-reduced
echelon form $A$ for $U$.

\begin{theorem} 
\label{qtsubspace}
If $q$ is a prime power then
\begin{equation}
\qbin{n}{k}{q,t} = \sum_{U} t^{s(U)},
\end{equation}
in which the summation runs over all $k$-dimensional $\FF_q$ subspaces $U$ of the $n$-dimensional
$\FF_q$-vector space $V$.
\end{theorem}
\begin{proof}
This is simply a reformulation of Theorem~\ref{box-interpretation-theorem}.
When the parametrization entry $a_{ij}$ of $A$ corresponds bijectively to the cell $x$ of
$\lambda$, then one has $i=\row(x)$, $\dist_U(i,j)=\dist(x)$, and
$(i,j)$ is the lowest parametrization entry in its column of $A$
if and only if $x$ lies at the bottom of its column of $\lambda$.
From this and the definition \eqref{lambda-weight-definition},
it easily follows that  $\wt(\lambda,k;q,t)=\sum_{U} t^{s(U)}$ as $U$ ranges over all
subspaces whose echelon form corresponds to $\lambda$.
\end{proof}

\section{Generalization 1:  principally specialized Schur functions}
\label{schur-function-section}

We review here some of the definition and properties of Macdonald's {\it $7{th}$ variation on Schur
functions}, and then lift them from polynomials with $\FF_q$ coefficients
to elements of the ring $\Qq(\xx)$. We then show (Corollary~\ref{skewschurpos}) 
that,
for integers $q \geq 2$, the principal specializations of these rational functions are actually polynomials in 
a variable $t$ with nonnegative integer coefficients, generalizing the $(q,t)$-binomials.

\subsection{Lifting Macdonald's finite field Schur polynomials}

 Macdonald's $7^{th}$ variation on Schur functions from \cite[\S 7]{Macdonald} are
elements of $\FF_q[\xx]:=\FF_q[x_1,\ldots,x_n]$ defined as follows.
For each $\alpha=(\alpha_1,\ldots,\alpha_n) \in \NN^n$, first define antisymmetric 
polynomials
$$
A_\alpha(\xx):= \det( x_i^{q^{\alpha_j}} )_{i,j=1}^n.
$$
Let $\delta_n:=(n-1,n-2,\ldots,1,0)$.  Given a partition 
$\lambda=(\lambda_1 \geq \cdots \geq \lambda_n)$, define
$$
\MacSchur_\lambda(\xx):= \frac{A_{\lambda+\delta_n}(\xx)}{A_{\delta_n}(\xx)},
$$
which is a priori only a rational function in $\FF_q(\xx)$, but
which Macdonald shows is actually a polynomial lying in $\FF_q[\xx]$.
Since both $A_{\lambda+\delta_n}(\xx)$ and $A_{\delta_n}(\xx)$ are antisymmetric polynomials,
their quotient $\MacSchur_\lambda(\xx)$ is a symmetric polynomial in the
$x_i$.  Macdonald shows that $\MacSchur_\lambda(\xx)$ enjoys the stronger property
of being invariant under the entire general linear group $G=GL_n(\FF_q)$.

\begin{definition}
Recall from Section~\ref{limits-section} that
$\Qq(\xx)=\Qq(\ttt)^{\otimes n}$ where $\Qq(\ttt)$ was defined there,
with Frobenius automorphism $\varphi$ acting by 
$(\varphi f)(x_i)=f(x_1^q,\ldots,x_n^q)$.
First define  
$$
\AZ_\alpha(\xx):= \det( x_i^{q^{\alpha_j}} )_{i,j=1}^n
$$
regarded as an element of $\Qq(\xx)$, and then define
$$
\MacSchurZ_\lambda(\xx):= \frac{\AZ_{\lambda+\delta_n}(\xx)}{\AZ_{\delta_n}(\xx)}
$$
in $\Qq(\xx)$.  Even after specializing to integers $q \geq 2$
this will turn out {\it not} to be a polynomial in general.  
We will be interested later in its {\it principal specialization}, lying
in $\Qq(\ttt)$:
\begin{equation}
\label{bialternant-product-formula}
\begin{aligned}
\MacSchurZ_\lambda(1,t,t^2,\ldots,t^{n-1})
 & = \frac{ \det \left( (t^{i-1})^{q^{\lambda_j+n-j}} \right)_{i,j=1}^n }{ \det \left( (t^{i-1})^{q^{n-j}} \right)_{i,j=1}^n }\\
 & = \frac{ \det \left( (t^{q^{\lambda_j+n-j}})^{i-1} \right)_{i,j=1}^n }{ \det \left( (t^{q^{n-j}})^{i-1} \right)_{i,j=1}^n }\\
 & = \prod_{0 \leq i < j \leq n-1}
       \frac{t^{q^{\lambda_{n-j}+j}}-t^{q^{\lambda_{n-i}+i}}}{t^{q^j}-t^{q^i}}.
\end{aligned}
\end{equation}
Here the last equality used the Vandermonde determinant formula.
\end{definition}

Define the analogue of the complete homogeneous symmetric function $h_r$ by
\begin{equation}
\label{H-definition}
\HZ_r(\xx)=\MacSchurZ_{(r)}(\xx)
\end{equation}
for $r \geq 0$, and $\HZ_r(\xx)=0$ for $r < 0$.

\begin{theorem} 
\label{qthomo}
For integers $n \geq k \geq 0$,
$$
\HZ_{(n-k)}(1,t,\cdots,t^k)=\qbin{n}{k}{q,t}.
$$
\end{theorem} 
\begin{proof} 
When $\lambda=(\lambda_0,\ldots,\lambda_k)=(n-k,0,\ldots,0)$, 
the last product in \eqref{bialternant-product-formula} 
$$
\MacSchurZ_\lambda(1,t,\ldots,t^k) 
  = \prod_{0 \leq i < j \leq k}     
       \frac{t^{q^{\lambda_{k+1-j}+j}}-t^{q^{\lambda_{k+1-i}+i}}}{t^{q^j}-t^{q^i}}
$$
will have the numerator exactly matching the denominator in all of its factors except those with $j=k$.
This leaves only these factors:
$$
\prod_{0 \leq i < k}     
       \frac{t^{q^{(n-k)+k}}-t^{q^{0+i}}}{t^{q^k}-t^{q^i}}\\
  = \qbin{n}{k}{q,t}.
$$
\end{proof}

\subsection{Jacobi-Trudi formulae}

Macdonald also proved Jacobi-Trudi-style determinantal formulae 
for his polynomials $\MacSchur_\lambda$, allowing him to generalize
them to skew shapes $\lambda/\mu$.  We review/adapt
his proof here so 
that it lifts to $\Qq(\xx)$, giving the same formula 
for $\MacSchurZ_\lambda(\xx)$.  

\begin{theorem} (Jacobi-Trudi) 
\label{qtJT}
For any partition $\lambda=(\lambda_1,\ldots,\lambda_n)$
$$
\MacSchurZ_{\lambda}(\xx)
 =det(\varphi^{1-j} \HZ_{\lambda_i-i+j}(\xx)_{i,j=1}^n )
$$
\end{theorem}
\begin{proof} (cf. \cite[\S 7]{Macdonald}) 
For convenience of notation, 
we omit the variable set $\xx$ from the notation for $\HZ, \AZ, \MacSchurZ_\lambda$ etc.
The definition of $\HZ_r$
says that
\begin{equation}
\begin{aligned}
\HZ_r := \MacSchurZ_{(r)} 
         &=\AZ_{\delta_n}^{-1} \cdot \det \left[ \begin{matrix}
                        x_1^{q^{n+r-1}} & x_1^{q^{n-2}} & x_1^{q^{n-3}}& \cdots & x_1^{q^1} & x_1 \\
                        x_2^{q^{n+r-1}} & x_2^{q^{n-2}} & x_2^{q^{n-3}}& \cdots & x_2^{q^1} & x_2 \\
                        \vdots      & \vdots    & \vdots   & \ddots & \vdots & \vdots \\
                        x_n^{q^{n+r-1}} & x_n^{q^{n-2}} & x_n^{q^{n-3}}& \cdots & x_n^{q^1} & x_n \\
                 \end{matrix} \right] \\
\end{aligned}
\end{equation}
Expanding the determinant along the first column shows that
\begin{equation}
\label{indep-of-r-equation}
\HZ_r(\xx)  =\sum_{k=1}^n  \varphi^{n+r-1}(x_k) \cdot u_k
\end{equation}
where $u_k$ in $\Qq(\xx)$ do not depend on $r$. For
$\alpha=(\alpha_1,\ldots,\alpha_n)$, write down equation \eqref{indep-of-r-equation}
for $r=\alpha_i - n +j$ for $i=1,2,\ldots,n$, and apply $\varphi^{1-j}$ for
$j=1,2,\ldots,n$, giving the $n \times n$ system of equations
$$
\varphi^{1-j} \HZ_{\alpha_i - n + j} = \sum_{k=1}^n \varphi^{\alpha_i}(x_k) \cdot \varphi^{1-j} u_k.
$$ 
Reinterpret this as a matrix equality:
$$
\left( \varphi^{1-j} \HZ_{\alpha_i - n + j} \right)_{i,j=1}^n
  = \left( \varphi^{\alpha_i}(x_k) \right)_{i,k=1}^n
      \cdot  \left( \varphi^{1-j} u_k \right)_{k,j=1}^n
$$
Taking the determinant of both sides yields
\begin{equation}
\label{det-evaluation-trick}
\det \left( \varphi^{1-j} \HZ_{\alpha_i - n + j} \right)_{i,j=1}^n
= \AZ_\alpha \cdot B
\end{equation}
where $B:= \det \left( \varphi^{1-j} u_k \right)_{k,j=1}^n$.
One pins down the value of $B$ by choosing
$\alpha=\delta_n$ in \eqref{det-evaluation-trick}:
then $\alpha_i -n +j=j-i$, so the left side becomes the determinant
of an upper unitriangular matrix, giving
$$
1 = \AZ_{\delta_n} B
$$
and hence $B=\AZ_{\delta_n}^{-1}$.  Thus \eqref{det-evaluation-trick} becomes
$$
\det \left( \varphi^{1-j} \HZ_{\alpha_i - n + j} \right)_{i,j=1}^n = \frac{\AZ_\alpha}{\AZ_{\delta_n}}.
$$
Now taking $\alpha=\lambda+\delta_n$ yields the theorem.
\end{proof}

Theorem~\ref{qtJT} 
shows that $\MacSchurZ_{\lambda}(\xx)$ is the special case when $\mu = \varnothing$ of
the following ``skew'' construction.

\begin{definition}
Given two partitions 
$$
\begin{aligned}
\lambda&=(\lambda_1,\ldots,\lambda_{\ell})\\
\mu&=(\mu_1,\ldots,\mu_\ell)
\end{aligned}
$$
with $\mu_i \leq \lambda_i$ for $i=1,2,\ldots,\ell$,
the set difference
of their Ferrers diagrams $\lambda/\mu$ is called a {\it skew shape}.
Define
\begin{equation}
\label{skew-definition}
\MacSchurZ_{\lambda/\mu}(\xx):= \det\left( \varphi^{\mu_j-(j-1)} \HZ_{\lambda_i-\mu_j-i+j} \right)_{i,j=1,2,\ldots,\ell}
\end{equation}

In the special case where  $\mu=\varnothing$ and 
$\lambda$ is a single row of size $r$,
note that $\MacSchurZ_{(r)}(\xx)=\HZ_r(\xx)$, consistent
with \eqref{H-definition}.

In the special case where $\mu=\varnothing$ and 
$\lambda=1^r$ is a single column of size $r$, define 
$$
\EZ_r(\xx):=\MacSchurZ_{(1^r)}(\xx),
$$
and set $\EZ_r(\xx):=0$ for $r < 0$.  The following analogue of Theorem~\ref{qthomo}
is proven analogously:
\begin{equation}
\label{qtelem}
\EZ_r(1,t,t^2,\ldots,t^{n-1}) 
  = \qbin{n}{n-r}{q,t} 
      \prod_{i=1}^r \frac{ t^{q^n}-t^{q^{n-r+i-1}} }{ t^{q^{n-r+i}}-t^{q^{n-r}}  }.
\end{equation}

\end{definition}

In \cite[\S9]{Macdonald} Macdonald explains how in a more general setting, one can write
down a dual Jacobi-Trudi determinantal formula for $\MacSchurZ_{\lambda/\mu}(\xx)$, involving
the {\it conjugate} or {\it transpose} partitions $\lambda', \mu'$ and the $\EZ_r$ instead of
the $\HZ_r$.

\begin{proposition}
\label{dual-Jacobi-Trudi}
$$
\MacSchurZ_{\lambda/\mu}(\xx):= 
 \det\left( \varphi^{-\mu'_j+j-1} \EZ_{\lambda'_i-\mu'_j-i+j}(\xx) \right)_{i,j=1,2,\ldots,\ell}.
$$

\end{proposition}
\begin{proof}
We sketch Macdonald's proof from \cite[\S9]{Macdonald} 
for completeness, and again, omit the variables $\xx$ from the notation for convenience.
We first prove that, for any interval $I$ in $\ZZ$, the two matrices
$$
\begin{aligned}
H_I &:= \left( \varphi^{1-j} \HZ_{j-i} \right)_{i,j \in I} \\
E_I &:= \left( (-1)^{j-i} \varphi^{-i} \EZ_{j-i} \right)_{i,j \in I}
\end{aligned}
$$
are inverses of each other.  Since both $H_I, E_I$ are upper unitriangular, this will
follow if one checks that for $i,k \in I$ with $i<k$ one has
$$
\sum_{j=i}^k \varphi^{1-j} \HZ_{j-i} \cdot (-1)^{k-j} \varphi^{-j} \EZ_{k-j} = 0,
$$
or equivalently, after applying $\varphi^i$ and re-indexing $k-i=:r$,
$$
\sum_{j=0}^r (-1)^j \varphi^{1-j} \HZ_{j} \cdot \varphi^{-j} \EZ_{r-j} = 0.
$$
But this is exactly what one gets from expanding along its top row the determinant in this definition:
$$
\EZ_r= \det \left( \phi^{1-j} \HZ_{1-i+j} \right)_{i,j=1}^r.
$$

Once one knows that $H_I, E_I$ are inverses of each other, and since they have
determinant $1$ by their unitriangularity, it follows that each minor subdeterminant
of $H_I$ equals the complementary cofactor of the transpose of $E_I$.  One can
check that for each $\lambda/\mu$ there is a choice of the interval $I$ and the appropriate
subdeterminant so that this equality is the one asserted in the proposition;
see \cite[Chap. 1, eqn. (2.9)]{Macdonald-book}.
\end{proof}

Our goal in the next section is to describe a combinatorial
interpretation for $\MacSchurZ_{\lambda}(1,t,\cdots,t^n)$ in terms of tableaux, which
will imply that for integers $q \geq 2$ this polynomial in $t$ has nonnegative coefficients.

\subsection{Nonintersecting lattice paths, and the weight of a tableau}
\label{lattice-paths-section}

  The usual skew Schur function $s_{\lambda/\mu}(\xx)$ has the following
well-known combinatorial interpretation (see e.g.\cite[\S I.5]{Macdonald}, 
\cite[\S7.10]{Stanley-EC2}):
\begin{equation}
\label{usual-Schur-function-interpretation}
s_{\lambda/\mu}(x_1,\ldots,x_n)=
 \sum_{T} \xx^T
\end{equation}
where $T$ runs over all reverse column-strict tableaux of shape $\lambda/\mu$ with
entries in $\{1,2,\ldots,n\}$.  Here a {\it reverse column-strict tableaux $T$} is an
assignment of an entry to each cell of the skew shape $\lambda/\mu$ in such a way
that the entries decrease weakly left-to-right in each row, and decrease strictly from
top-to-bottom in each column.  The monomial $\xx^T:=\prod_i x_{T_i}$ as $i$ ranges over
the cells of $\lambda/\mu$.  

There is a well-known combinatorial proof of this formula 
(see \cite[Theorem 2.7.1]{Stanley-EC1} and \cite[\S7.16]{Stanley-EC2}) 
due to Gessel and Viennot.  This proof begins with the 
Jacobi-Trudi determinantal expression for $s_{\lambda/\mu}$,
reinterprets this as a signed sum over tuples of lattice paths, cancels this sum
down to the nonintersecting tuples of lattice paths, and then shows how these biject
with the tableaux.

Note that \eqref{usual-Schur-function-interpretation} implies the following
interpretation for the principal specialization of $s_{\lambda/\mu}$:
$$
s_{\lambda/\mu}(1,t,\ldots,t^k)= \sum_{T} t^{\sum_i T_i}
$$
where $T$ ranges over the reverse column-strict tableaux of shape $\lambda/\mu$ with
entries in $\{0,1,\ldots,k\}$. 
Our goal is to generalize this to $\MacSchurZ_{\lambda/\mu}(1,t,\ldots,t^k)$,
using the Gessel-Viennot proof mentioned above.

We begin by recalling how lattice paths biject with partitions and
tableaux, in order to put the appropriate weight on the tuples
of lattice paths.  We start with the easy bijections between these three objects:
\begin{enumerate}
\item[(i)] Partitions $\nu$ inside a $k \times r$ rectangle.
\item[(ii)] Lattice paths $P$ taking unit steps north $(N)$ and east $(E)$
from $(x,y)$ to $(x+r,y+k)$.
\item[(iii)] Reverse column-strict tableaux of the single row shape $(r)$
and entries in $\{0,1,\ldots,k\}$.
\end{enumerate}
The bijection between (i) and (ii) sends the lattice path $P$ to the Ferrers diagram $\nu(P)$ 
(in English notation) having $P$ as its outer boundary and northwest corner at $(x,y+k)$.  The
bijection between (ii) and (iii) sends the lattice path $P$ to the tableau whose entries
give the depths below the line $y=k$ of the horizontal steps in the path $P$.  For example, if $k=4, r=5$
then the partition $\nu=(4,4,1,0)$ corresponds to the path $P$ whose unit steps form the
sequence $(N,E,N,E,E,E,N,N,E)$, which corresponds to the single-row tableau $T=32220$;  
see Figure~\ref{tableau-figure}(a).

\begin{figure}
\label{tableau-figure}
\epsfysize = 3.0 in 
\centerline{\epsffile{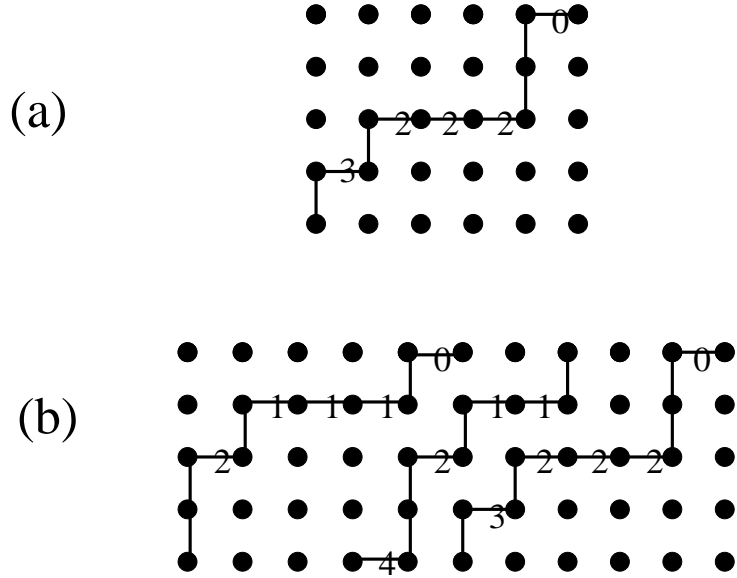}} 
\caption{
(a) The lattice path corresponding to the partition 
$\nu=(4,4,1,0)$ inside a $4 \times 5$ rectangle,
or to the single-row tableau $T=32220$ with entries in $\{0,1,2,3,4\}$.  (b) From right to left, 
the $3$-tuple $(P_1,P_2,P_3)$ of nonintersecting lattice paths corresponding to the
tableau $T$ described in the text.}
\end{figure}

Given any skew shape $\lambda/\mu$, the bijection between (ii) and (iii) above
generalizes to one between these two sets:
\begin{enumerate}
\item[$\bullet$] All $\ell$-tuples $(P_1,\ldots,P_\ell)$ of lattice paths,
where $P_i$ goes from $(\mu_i-(i-1),0)$  to $(\lambda_i-(i-1),k)$,
and no pair $P_i,P_j$ of paths touches, that is, the paths are {\it nonintersecting}.
\item[$\bullet$] Reverse column strict tableaux of shape $\lambda/\mu$ with
entries in $\{0,1,\ldots,k\}$.
\end{enumerate}
Here the bijection has the $i^{th}$ row of the tableaux giving the depths below the line
$y=k$ of the horizontal steps in the $i^{th}$ path $P_i$.  For example, if $k=4$ and $\ell=3$ with
$\lambda/\mu=(8,6,5)/(3,2,0)$, then the tableau $T$ of shape $\lambda/\mu$ having
entries in $\{0,1,2,3,4\}$ given by
$$
T=
\begin{matrix}
\cdot & \cdot & \cdot & 3 & 2 & 2 & 2& 0\\
\cdot & \cdot & 4     & 2 & 1& 1 &  &  \\
2     & 1     & 1     & 1 & 0&   &  & 
\end{matrix}
$$
corresponds to the $3$-tuple $(P_1,P_2,P_3)$ where 
$$
\begin{aligned}
&P_1\text{ is the path from }(3,0)\text{ to }(8,4)\text{ with steps }(N,E,N,E,E,E,N,N,E)\\
&P_2\text{ is the path from }(1,0)\text{ to }(5,4)\text{ with steps }(E,N,N,E,N,E,E,N)\\
&P_3\text{ is the path from }(-2,0)\text{ to }(3,4)\text{ with steps }(N,N,E,N,E,E,E,N,E)
\end{aligned}
$$
as depicted from right to left in Figure~\ref{tableau-figure}(b).

Given such a tableau $T$ corresponding to the tuple of paths 
$(P_1,\ldots,P_\ell)$, and letting $\nu(P_i)$ be the partition
within a rectangle of width $k$ that has $P_i$ as its outer boundary as before, 
define
\begin{equation}
\label{tableau-weight-definition}
\wt(T;q,t):=
   \prod_{i=1}^\ell \varphi^{\mu_i-(i-1)} \wt(\nu(P_i),k;q,t).
\end{equation}
In the next proof, we will use the fact that for a lattice path $P$ and a
cell $x$ of $\nu(P)$, the distance $\dist_P(x)$ which
appears in the formula \eqref{lambda-weight-definition} 
for $\wt(x,\nu(P),k;q,t)$ is the
taxicab distance from the starting point of path $P$ to the cell $x$.

\begin{theorem} 
\label{Schurpos} 
$$
\MacSchurZ_{\lambda/\mu}(1,t,\cdots,t^k)
  = \sum_T \wt(T;q,t)
$$
where the sum ranges over all reverse column-strict tableaux $T$ of shape $\lambda/\mu$ with
entries in $\{0,1,\ldots,k\}$.

Furthermore,
\begin{equation}
\lim_{t \rightarrow 1} \wt(T; q,t) = q^{\sum_i T_i}
\quad \text{ and } \quad
\lim_{q \rightarrow 1} \wt(T; q,t^{\frac{1}{q-1}}) = t^{\sum_i T_i}.                                                                                           
\end{equation}
\end{theorem}

\begin{proof} 
We recapitulate and adapt the usual Gessel-Viennot proof alluded to above,
being careful to ensure that the weights behave correctly in this context.

Starting with the definition \eqref{skew-definition}, and using
Theorems~\ref{qthomo} and \ref{box-interpretation-theorem}
to replace occurrences of $\HZ_r$, one has 
$$
\begin{aligned}
\MacSchurZ_{\lambda/\mu}(1,t,\cdots,t^k)
  &= \det( \varphi^{\mu_j-(j-1)} \HZ_{\lambda_i-\mu_j-i+j} )_{i,j=1,2,\ldots,\ell} \\
  &= \sum_{w \in \Symm_\ell} \sgn(w) \prod_{i=1}^\ell  
       \varphi^{\mu_{w(i)}-(w(i)-1)} \HZ_{\lambda_i-\mu_{w(i)}-i+w(i)}(1,t,\ldots,t^k) \\
  &= \sum_{(w,(P_1,\ldots,P_\ell))} \sgn(w) \prod_{i=1}^\ell 
       \varphi^{\mu_{w(i)}-(w(i)-1)} \wt(\nu(P_i),k;q,t).
\end{aligned}
$$
The last summation runs over pairs $(w,(P_1,\ldots,P_\ell))$, where $w$ is a
permutation in $\Symm_\ell$, and $(P_1,\ldots,P_\ell)$ is an $\ell$-tuple of lattice paths
where $P_i$ goes from $(\mu_{w(i)}-(w(i)-1),0)$ to $(\lambda_i-(i-1),k)$.

One wants to cancel all the terms in the last sum above that have at least one pair $(P_i,P_j)$ of 
intersecting lattice paths.  In particular, this occurs if $w$ is not the 
identity permutation.  Therefore all terms remaining will have $w$ equal to the
identity, and the paths $(P_1,\ldots,P_\ell)$ nonintersecting.  Note that these
leftover paths have the correct weight $\wt(T;q,t)$ for their corresponding
tableau $T$, as given in the theorem.

The Gessel-Viennot cancellation argument involves tail-swapping.  
One cancels a term $(w,(P_1,\ldots,P_\ell))$ 
with another term having equal weight and opposite sign.  
One finds this other term, say by choosing
the southeasternmost intersection point $p$ for any pair of the paths, and choosing
the lexicographically smallest pair of indices $i<j$ of paths $P_i,P_j$ which touch at $p$.
Then replace $w$ by $w':=w \cdot (i,j)$, and replace the pair of paths $(P_i,P_j)$ with the pair
of paths $(P'_j,P'_i)$, in which $P'_i$ (resp. $P'_j$) follows the path $P_i$ up until it reaches $p$,
but then follows $P_j$ (resp. $P_i$) from that point onward.  

One must check that this replacement does not change the weight, that is,
\begin{equation}
\label{tail-swapping-equation}
\begin{aligned}
&\varphi^{\mu_{w(i)}-(w(i)-1)} \wt(\nu(P_i),k;q,t) 
 \cdot \varphi^{\mu_{w(j)}-(w(j)-1)} \wt(\nu(P_j),k;q,t) \\
& =
  \varphi^{\mu_{w'(i)}-(w'(i)-1)} \wt(\nu(P'_j),k;q,t) 
  \cdot \varphi^{\mu_{w'(j)}-(w'(j)-1)} \wt(\nu(P'_i),k;q,t) \\
& \quad =
  \varphi^{\mu_{w(i)}-(w(i)-1)} \wt(\nu(P'_i),k;q,t) 
   \cdot \varphi^{\mu_{w(j)}-(w(j)-1)} \wt(\nu(P'_j),k;q,t).
\end{aligned}
\end{equation}
This follows by comparing how each cell $x$ of $\nu(P_i)$
(or of $\nu(P_j)$) contributes a factor of the form
$\varphi^m \wt(x,\nu(P),k;q,t)$ to the products on the left
and right sides of \eqref{tail-swapping-equation}.
There are two cases. If $x$ is a cell of either $\nu(P_i)$
or $\nu(P_j)$ lying to the left of the point $p$, then it will
contribute the same factor on both sides.  
If $x$ is a cell of $\nu(P_i)$ lying to the right of $p$, then it contributes 
$$
\begin{aligned}
&\varphi^{\mu_{w(i)}-(w(i)-1)} \wt(x,\nu(P_i),k;q,t) 
  \text{ on the leftmost side of }\eqref{tail-swapping-equation}\text{ and }\\
&\varphi^{\mu_{w(j)}-(w(j)-1)} \wt(x,\nu(P'_j),k;q,t) 
  \text{ on the rightmost side of }\eqref{tail-swapping-equation}.\\
\end{aligned}
$$
However, we claim there is an equality
$$
\varphi^{\mu_{w(i)}-(w(i)-1)} \wt(x,\nu(P_i),k;q,t)
= \varphi^{\mu_{w(j)}-(w(j)-1)} \wt(x,\nu(P'_j),k;q,t).
$$
Definition \eqref{lambda-weight-definition} shows
that the exponent $e_k(x):=q^{\row(x)+\dist_P(x)}-q^{\dist_P(x)}$
determining the powers of $t$ appearing in $\wt(x,\nu(P),k;q,t)$
will be computed with the {\it same} row index $\row(x)$,
whether one considers $x$ inside $\nu(P_i)$ or inside $\nu(P'_j)$.
However, the distance $\dist_{P_i}(x)$ from $x$ to the start of $P_i$
is $\mu_{w(j)}-\mu_{w(i)}-w(i)+w(j)$ {\it less} than its distance $\dist_{P'_j}(x)$ to the start of 
$P'_j.$ This difference is exactly compensated by the difference in 
the Frobenius power which will be applied:
$\varphi^{\mu_{w(i)}-(w(i)-1)}$ versus $\varphi^{\mu_{w(j)}-(w(j)-1)}$.
The argument is similar for a cell of $\nu(P_j)$ lying to the right of $p$.
Thus the two terms have the same weight, and opposite signs: $\sgn(w')=-\sgn(w)$.

One can check that this bijection between terms is an involution, and
hence it provides the necessary cancellation.
\end{proof}

\begin{corollary}
\label{skewschurpos}
If $q\ge 2$ is an integer, then $\MacSchurZ_{\lambda}(1,t,\cdots,t^k)$ 
is a polynomial in $t$ with 
nonnegative coefficients.
\end{corollary}
\begin{proof}
Use Theorem~\ref{Schurpos}.  We wish to show that
when $\mu=\varnothing$,  for every column-strict tableau $T$ of
shape $\lambda=\lambda/\mu$, every
factor in the product formula for $\wt(T;q,t)$ given by
\eqref{tableau-weight-definition} is a polynomial in $t$
with nonnegative coefficients for $q \geq 2$ an integer.

The key point is that $\mu=\varnothing$ means the sequence
of nonintersecting lattice paths $(P_1,P_2,\ldots,P_\ell)$
corresponding to $T$ will start in the {\it consecutive} 
positions 
$$
(0,0),(-1,0),(-2,0), \ldots, (-\ell+1,0).
$$
For each $i$, this forces the path $P_i$ 
to start with at least $i-1$ vertical steps, in order to
avoid intersecting the next path $P_{i-1}$.  Hence the
corresponding partition shape $\nu(P_i)$ will be
missing at least $i-1$ boxes in its first
column.  Consequently, in the product formula for $\wt(T;q,t)$ given by
\eqref{tableau-weight-definition}, one can rewrite each factor
$$
\begin{aligned}
\varphi^{\mu-(i-1)} \wt(\nu(P_i),k;q,t) 
&= \varphi^{-(i-1)} \wt(\nu(P_i),k;q,t)\\
&=\wt(\nu(P_i),k-(i-1);q,t)
\end{aligned}
$$
where the last equality uses $i-1$ times repeatedly 
the second case of the recurrence in 
Proposition~\ref{lambda-weight-recurrence}.
The formula for $\wt(\nu(P_i),k-(i-1);q,t)$
given in \eqref{lambda-weight-definition} then
shows that it is a polynomial in $t$ with
nonnegative coefficients whenever $q \geq 2$ is an integer.
\end{proof}

\begin{example}
The two skew shapes
$$
(2,1)=(2,1)/(0,0) =   \begin{matrix} \times & \times \\
                   \times &     
    \end{matrix} 
\qquad \text{ and }\qquad
(2,2)/(1,0)=    \begin{matrix} \cdot & \times \\
                   \times     & \times 
    \end{matrix} \qquad
$$
have the same {\it ordinary} Schur functions, and hence the same principal specializations
$$
s_{(2,1)}(1,q) = s_{(2,2)/(1,0)}(1,q)=q^1+q^2
$$
corresponding to either the 
two reverse column-strict tableaux
$$
T_1=\begin{matrix} 1 & 0 \\
                   0 &    
    \end{matrix} \qquad
T_2=\begin{matrix} 1 & 1 \\
                   0 &    
    \end{matrix} \qquad \text{ of shape }(2,1)
$$
or the tableaux
$$
T'_1=\begin{matrix} \cdot & 1 \\
                   0     & 0 
    \end{matrix} \qquad
T'_2=\begin{matrix} \cdot & 1 \\
                   1     & 0   
    \end{matrix} \qquad \text{ of shape }(2,2)/(1,0).
$$

We compare here what the preceding results say for
$$
\begin{aligned}
\MacSchurZ_{(2,1)/(0,0)}(1,t) = \MacSchurZ_{(2,1)}(1,t)
&=
\det \left[ 
\begin{matrix}
\varphi^0 \HZ_{2}(1,t) & \varphi^{-1}\HZ_{3}(1,t) \\
\varphi^0 \HZ_{0}(1,t) & \varphi^{-1}\HZ_{1}(1,t) 
\end{matrix}
\right]  \\
&=[1+q+q^2]_{t^{q-1}} \cdot [1+q]_{t^\frac{q-1}{q}} -
    [1+q+q^2+q^3]_{t^\frac{q-1}{q}} \\
&=t+t^2+t^3+t^4+t^5+t^6 \text{ when }q=2,
\end{aligned}
$$
versus
$$
\begin{aligned}
\MacSchurZ_{(2,2)/(1,0)}(1,t) 
&=
\det \left[ 
\begin{matrix}
\varphi^{1} \HZ_{1}(1,t) & \varphi^{-1}\HZ_{3}(1,t) \\
\varphi^{1} \HZ_{0}(1,t) & \varphi^{-1}\HZ_{2}(1,t) 
\end{matrix}
\right]  \\
&=[1+q]_{t^{q(q-1)}} \cdot [1+q+q^2]_{t^\frac{q-1}{q}} -
    [1+q+q^2+q^3]_{t^\frac{q-1}{q}} \\
&=t^2+t^{\frac{5}{2}}+t^3+t^4+t^{\frac{9}{2}}+t^5 \text{ when }q=2.
\end{aligned}
$$
Note that $\MacSchurZ_{(2,1)}(1,t) \neq \MacSchurZ_{(2,2)/(1,0)}(1,t)$.
Both are elements of $\Qq(\ttt)$ and can be rewritten as weighted sums over
tableaux, according to Theorem~\ref{Schurpos}:
$$
\begin{aligned}
\MacSchurZ_{(2,1)/(0,0)}(1,t) 
&=\wt(T_1;q,t) + \wt(T_2;q,t)\\
&=t^{q-1}[q]_{t^{q-1}} + t^{q-1}[q]_{t^{q-1}} \cdot t^{q^2-q}[q]_{t^{q^2-q}}\\
&=(t+t^2) + (t^3+t^4+t^5+t^6)\text{ when }q=2,\\
 & \text{ versus } \\
\MacSchurZ_{(2,2)/(1,0)}(1,t) 
&=\wt(T'_1;q,t) + \wt(T'_2;q,t)\\
&= t^{q^2-q}[q]_{t^{q^2-q}} + t^{q^2-q}[q]_{t^{q^2-q}} \cdot t^{1-\frac{1}{q}}[q]_{t^{1-\frac{1}{q}}} \\
&=(t^2+t^4) + (t^{\frac{5}{2}}+t^3+t^{\frac{9}{2}}+t^5)\text{ when }q=2.
\end{aligned}
$$
Lastly, note that that $\MacSchurZ_{(2,1)}(1,t)$ is a polynomial in $t$ 
(with nonnegative coefficients) for integers $q \geq 2$, as predicted by
Proposition~\ref{dual-Jacobi-Trudi}, but this is not true for the skew
example $\MacSchurZ_{(2,2)/(1,0)}(1,t)$. 
\end{example}

\section{Generalization 2: multinomial coefficients}
\label{multinomial-section}

  We explore here a different generalization of the $(q,t)$-binomial coefficient, this
time to a multinomial coefficient that appears naturally within the invariant theory
of $GL_n(\FF_q)$.

\subsection{Definition}

Given a composition $\alpha=(\alpha_1,\ldots,\alpha_\ell)$ of $n$,
define its partial sums $\sigma_s:=\alpha_1+\alpha_2+\cdots+\alpha_s$, so that $\sigma_0=0$ and
let
$$
\qbin{n}{\alpha}{q,t}:= \qbin{n}{\alpha_1,\ldots,\alpha_\ell}{q,t} 
:=\frac{n!_{q,t}}{\alpha !_{q,t}}
 = \frac{ \prod_{i=1}^n (1-t^{q^n-q^{n-i}}) }
     { \prod_{s=1}^{\ell} \prod_{i=1}^{\alpha_s} (1-t^{q^{\sigma_s}-q^{\sigma_s-i}}) },
$$
where
$$
\alpha!_{q,t}:=\alpha_1!_{q,t} \, \cdot \, \alpha_2!_{q,t^{q^{\sigma_1}}}  \, \cdot \, 
                \alpha_3!_{q,t^{q^{\sigma_2}}}  \, \cdots  \, \alpha_\ell!_{q,t^{q^{\sigma_{\ell-1}}}}.
$$
Note its relation to the $(q,t)$-binomial
$$
\qbin{n}{k}{q,t}=\qbin{n}{k,n-k}{q,t},
$$
as well as these formulae:
\begin{equation}
\label{multinomial-in-terms-of-binomials}
\begin{aligned}
\qbin{n}{\alpha}{q,t}
&=\qbin{n}{\alpha_1}{q,t}
  \varphi^{\alpha_1} \qbin{n-\alpha_1}{\alpha_2,\ldots,\alpha_\ell}{q,t}\\
&=\qbin{n}{\alpha_1}{q,t}
  \varphi^{\sigma_1} \qbin{n-\alpha_1}{\alpha_2}{q,t}
  \varphi^{\sigma_2} \qbin{n-\alpha_1-\alpha_2}{\alpha_3}{q,t} 
\cdots
\end{aligned}
\end{equation}
Equations~\ref{multinomial-in-terms-of-binomials} along with Corollary~\ref{qtbinpos} imply
that for integers $q \geq 2$,
the $(q,t)$-multinomial is a polynomial in $t$ with nonnegative coefficients.  
As with the $(q,t)$-binomial, it has two limiting values
given by the usual {\it $q$-multinomial coefficient}:
$$
\begin{aligned}
\lim_{t\rightarrow 1}\qbin{n}{\alpha}{q,t} &=\qbin{n}{\alpha}{q}
  := \frac{ \prod_{i=1}^n \left( q^n-q^{n-i} \right) }
     { \prod_{s=1}^{\ell} \prod_{i=1}^{\alpha_s}\left( q^{\sigma_s}-q^{\sigma_s-i } \right)}\\
\lim_{q\rightarrow 1}\qbin{n}{\alpha}{q,t^{\frac{1}{q-1}}}&=\qbin{n}{\alpha}{t}.
\end{aligned}
$$

\subsection{Algebraic interpretation of multinomials}

We recall here two algebraic interpretations of the usual multinomial and
$q$-multinomial that were mentioned in the Introduction, and then give the
analogue for the $(q,t)$-multinomial.

The symmetric group $W=\Symm_n$ acts transitively on the collection of all flags of subsets 
$$
\varnothing=:S_0 \subset S_1 \subset \cdots \subset S_{\ell-1} \subset S_\ell:=\{1,2,\ldots,n\}
$$
in which $|S_i|= \sigma_i$.  The stabilizer of one such flag 
is the {\it Young} or {\it parabolic subgroup} $W_\alpha$ which
permutes separately the first $\alpha_1$ integers, the next $\alpha_2$ integers, etc.
Thus the coset space $W/W_\alpha$ is identified with the collection of these flags, and
hence has cardinality $[W:W_\alpha]=\binom{n}{\alpha}$.
When $q$ is a prime power, the $q$-multinomial analogously 
gives the cardinality of the finite partial flag manifold $G/P_\alpha$ where the group $G=GL_n(\FF_q)$
and $P_\alpha$ is the parabolic subgroup that stabilizes one of the flags of $\FF_q$-subspaces
$$
\{0\}=:V_0 \subset V_1 \subset \cdots \subset V_{\ell-1} \subset V_\ell:=\FF_q^n
$$ 
in which $\dim_{\FF_q} V_i = \sigma_i$.

On the other hand, one also has parallel Hilbert series interpretations
arising from the invariant theory of these groups acting on appropriate polynomial
algebras.  In the case of the $(q,t)$-multinomial this is where
its definition arose initially in work of the authors with D. White \cite[\S9]{RSW}.

Let $\ZZ[\xx]:=\ZZ[x_1,\ldots,x_n]$ carry its usual action of $W=\Symm_n$ by permutations 
of the variables, and let $\FF_q[\xx]:=\FF_q[x_1,\ldots,x_n]$ carry its usual action
of $G=GL_n(\mathbb{F}_q)$ by linear substitution of variables.  The {\it fundamental theorem of
symmetric functions} states that the invariant subring $\ZZ[\xx]^W$ is a polynomial algebra
generated by the elementary symmetric functions $e_1(\xx),\ldots,e_n(\xx)$.
A well-known theorem of Dickson (see e.g. \cite[\S 8.1]{Benson}) 
asserts that the invariant subring $\FF_q[\xx]^G$ is 
a polynomial algebra;  its generators can be chosen to be the {\it Dickson polynomials},
which are the same as Macdonald's polynomials $E_1(\xx),\ldots,E_n(\xx)$ 
discussed in Section~\ref{schur-function-section} above.  
It is also not hard to see that, for any composition $\alpha$ of $n$,
the invariant subring  $\ZZ[\xx]^{W_\alpha}$ is a polynomial algebra, whose
generators may be chosen as the elementary symmetric functions in the
first $\alpha_1$ variables, then those in the next $\alpha_2$ variables, etc.
The following result of Mui \cite{Mui} (see also Hewett \cite{Hewett})
is less obvious.

\begin{theorem} 
For every composition $\alpha$ of $n$,
the parabolic subgroup $P_\alpha$ 
has invariant subring $\FF_q[\xx]^{P_\alpha}$ isomorphic to a polynomial algebra 
$\mathbb{F}_q[f_1,\ldots,f_n]$.  Furthermore, the generators 
$f_1,\ldots,f_n$ may be chosen homogeneous with
degrees $q^{\sigma_s}-q^{\sigma_s-i}$ for
$s=1,\ldots,\ell$ and $i=1,\ldots, \alpha_s$.
\end{theorem}

\begin{corollary}
\label{free}
For every composition $\alpha$ of $n$,
\begin{equation}
\label{quotients}
\begin{array}{rcl}
\qbin{n}{\alpha}{t} &= \frac{\Hilb(\ZZ[\xx]^{W_\alpha},t)}{\Hilb(\ZZ[\xx]^W,t)} \\
\qbin{n}{\alpha}{q,t} &= \frac{\Hilb(\FF_q[\xx]^{P_\alpha},t)}{\Hilb(\FF_q[\xx]^G,t)}.
\end{array}
\end{equation}
Furthermore,
$\ZZ[\xx]^{W_\alpha}, \FF_q[\xx]^{P_\alpha}$ are free as modules over
over $\ZZ[\xx]^W, \FF_q[\xx]^{GL_n}$, respectively, and hence
\begin{equation}
\label{re-intepretations}
\begin{array}{rcl}
\qbin{n}{\alpha}{t} &= \Hilb(\ZZ[\xx]^{W_\alpha}/(\ZZ[\xx]^W_+),t) \\
\qbin{n}{\alpha}{q,t} &= \Hilb(\FF_q[\xx]^{P_\alpha}/( \FF_q[\xx]^G_+), t). 
\end{array}
\end{equation}
\end{corollary}
\begin{proof}
The Hilbert series for a graded polynomial algebra with generators in degrees
$d_1,\ldots,d_n$ is $\prod_{i=1}^n\frac{1}{1-t^{d_i}}$.  Hence the multinomial expressions
\eqref{quotients} for the quotient of Hilbert series follows in each case from consideration of the
degrees of the generators for $\ZZ[\xx]^W, \ZZ[\xx]^{W_\alpha}, \FF_q[\xx]^G, \FF_q[\xx]^{P_\alpha}$.

For the re-interpretations in \eqref{re-intepretations}, one needs a little
invariant theory and commutative algebra, such as can be found in
the book by Benson \cite{Benson} or the survey by Stanley \cite{Stanley-invariants}.
When two nested finite groups $H \subset G$ act on a Noetherian ring $R$, the invariant subring
$R^H$ is finitely generated as a module over $R^G$.  If $R^G$ is a polynomial
subalgebra, this means its generators will form a system of parameters for $R^H$.  When
$R^H$ is also polynomial, it is Cohen-Macaulay, and hence a free module over
the polynomial subalgebra generated by any system of parameters.  In this situation,
when the rings and group actions are all graded, a free basis for $R^H$ as an $R^G$-module 
can be obtained by lifting any basis
for $R^H/(R^G_+)$ as a module over the ring $R^G/R^G_+$ (which equals $\FF_q$ or $\ZZ$
in our setting).  Therefore
$$
\frac{\Hilb(R^H,t)}{\Hilb(R^G,t)}=\Hilb(R^H/(R^G_+),t).
$$
\end{proof}

\begin{remark}
It should be clear that the above Hilbert series interpretation of the $(q,t)$-multinomial
generalizes the ``$q=1$'' interpretation of the $t$-multinomial.  It turns out
that it also generalizes the interpretation of the $q$-multinomial as $[G:P_\alpha]=|G/P_\alpha|$ 
when $t=1$,
for the following reason:  when two nested finite subgroups $H \subset G \subset GL_n(k)$ act on $k[\xx]$,
one always has (see \cite[\S 2.5]{Benson})
$$
\lim_{t \rightarrow 1} \frac{\Hilb(k[\xx]^H,t)}{\Hilb(k[\xx]^G,t)} = [G:H].
$$
\end{remark}

\section{The $(q,t)$-analogues of $q^{\ell(w)}$}
\label{permutations-section}

  Corollary~\ref{free} interprets the $(q,t)$-multinomial algebraically when one specializes
$q$ to be a prime power.  Our goal here is a combinatorial interpretation valid in
general, generalizing Theorem~\ref{box-interpretation-theorem}.

  Given a composition $\alpha$ of $n$, recall that $W=\Symm_n$ has a parabolic subgroup $W_\alpha$
whose index $[W:W_\alpha]$ is given by the multinomial coefficient $\binom{n}{\alpha}$.
Viewing $W$ as a Coxeter group with the adjacent transpositions $(i,i+1)$ as
its usual set of Coxeter generators, one has its {\it length function}
$\ell(w)$ defined as the minimum length of $w$ as a product of these generators.
This is well-known to be the {\it inversion number}, counting pairs $1 \leq i < j \leq n$ with
$w(i) > w(j)$.  

  There are distinguished {\it minimum-length coset representatives}
$W^\alpha$ for $W/W_\alpha$:  a permutation $w$ lies in $W^\alpha$
if and only if 
$\alpha$ refines its {\it descent composition} 
$\beta(w)=(\beta_1,\ldots,\beta_\ell)$.
Recall that $\beta(w)$ is the composition of $n$ 
defined by the property that
$$
w(i) > w(i+1) \text{ if and only if }  
i \in 
\{\beta_1,\beta_1+\beta_2,\ldots,\beta_1+\beta_2+\cdots+\beta_{\ell-1}\}.
$$
Rephrased, $\beta(w)$ lists the lengths of the maximal increasing 
consecutive subsequences of  $w=(w(1),w(2),\ldots,w(n))$.

  It is well-known (see \cite[Prop. 1.3.17]{Stanley-EC1}) that 
\begin{equation}
\label{multinomial-interps}
\begin{aligned}
\binom{n}{\alpha} &= |W^\alpha| \\
\qbin{n}{\alpha}{q} &= \sum_{\substack{w \in W^\alpha}} q^{\ell(w)}.
\end{aligned}
\end{equation}

We wish to similarly express the $(q,t)$-multinomial as a sum over $w$ in 
$W^\alpha$ of a weight $\wt(w;q,t)$, simultaneously generalizing 
\eqref{multinomial-interps} and Theorem~\ref{compat}.  
The weight $\wt(w;q,t)$  will be defined recursively in a way that
generalizes the defining recurrence for $\wt(\lambda,k;q,t)$ in
Proposition~\ref{lambda-weight-recurrence} .  

Recall that when $\alpha=(k,n-k)$ has only two parts, there is a bijection 
$\lambda \leftrightarrow u_\lambda$ between partitions $\lambda$ inside a $k \times (n-k)$ rectangle and 
the minimum length coset representatives $u_\lambda$ in $W^\alpha$, determined by
$$
u_\lambda(i)=\lambda_{k+1-i}+i-1 \text{ for }i=1,2,\ldots,k.
$$
Now given any permutation $w$ in $W=\Symm_n$, if $k$ is defined by $w^{-1}(1)=k+1$, 
then taking $\alpha=(k,1,n-k-1)$, one can uniquely express
$$
w = u_\lambda \cdot a \cdot e \cdot b
$$
with 
$$
\ell(w) = \ell(u_\lambda) + \ell(a) + \ell(b)
$$ 
where 
\begin{enumerate}
\item[$\bullet$]
$u_\lambda \in W^{(k,n-k)}=\Symm_n^{(k,n-k)}$, 
\item[$\bullet$]
$a \in \Symm_{\{1,2,\ldots,k\}} \cong \Symm_k$,
\item[$\bullet$]
$e \in \Symm_{\{k+1\}} \cong \Symm_1$ (so $e$ is the identity permutation of $\{k+1\}$),
\item[$\bullet$]
$b \in \Symm_{\{k+2,k+3,\ldots,n\}} \cong \Symm_{n-k-1}$.
\end{enumerate}
Note also that in the above factorization $w=u_\lambda aeb$, by the definition of $k$,
one knows that $\lambda$ has its first column {\it full} of length $k$.  Let
$\hat\lambda$ denote the partition inside a 
$k \times (n-1-k)$ rectangle obtained from $\lambda$ by removing this first column,
so that $u_{\hat{\lambda}}$ lies in $\Symm_{n-1}^{(k,n-1-k)}$.

\begin{definition}
For $w \in \Symm_n$, define $\wt(w;q,t)$ in $\Qq(\ttt)$ {\it recursively} 
to be $1$ if $n=1$, and otherwise if $w^{-1}(1)=k+1$ set 
\begin{equation}
\label{first-w-recursion}
\wt(w;q,t) 
:= t^{q^k-1} \frac{k!_{q,t^q}}{k!_{q,t}} 
  \cdot \wt(u_{\hat{\lambda}}; q,t^q)  \wt(a;q,t)  \wt(b;q,t^{q^{k+1}})
\end{equation}
\end{definition}

\begin{example}
Let $n=8$ and choose
$$
w=\left( 
\begin{matrix} 
1 & 2 & 3 & 4 & 5 & 6 & 7 & 8 \\
5 & 2 & 7 & 4 & 1 & 3 & 8 & 6
\end{matrix}
\right)
$$
Then $k+1=w^{-1}(1)=5$, so that $k=4$, and the above factorization is
$$
\begin{aligned}
 w&=u_\lambda \cdot  a \cdot e \cdot  b\\
 &=\left(
\begin{matrix} 
1 & 2 & 3 & 4 & |&5 &|& 6 & 7 & 8 \\
2 & 4 & 5 & 7 & |&1 &|& 3 & 6 & 8
\end{matrix}
\right)
\cdot
\left(
\begin{matrix} 
1 & 2 & 3 & 4 \\
3 & 1 & 4 & 2  
\end{matrix}
\right)
\cdot
\left(
\begin{matrix} 
5 \\
5  
\end{matrix}
\right) 
\cdot
\left(
\begin{matrix} 
6 & 7 & 8 \\
6 & 8 & 7  
\end{matrix}
\right).
\end{aligned}
$$
Here $\lambda=(3,2,2,1)$, so that $\hat\lambda=(2,1,1,0)$ and
$$
u_{\hat\lambda}=
\left(
\begin{matrix} 
1 & 2 & 3 & 4 & |& 5 & 6 & 7 \\
1 & 3 & 4 & 6 & |& 2 & 5 & 7
\end{matrix}
\right).
$$
Then the recursive definition says
$$
\begin{aligned}
&\wt(w;q,t) \\
&:= t^{q^4-1} \frac{4!_{q,t^q}}{4!_{q,t}} 
  \cdot \wt(u_{\hat\lambda}; q,t^q)  \wt(a;q,t)  \wt(b;q,t^{q^{5}})\\
&:= t^{q^4-1} [q]_{t^{q^4-1}} [q]_{t^{q^4-q}} [q]_{t^{q^4-q^2}} [q]_{t^{q^4-q^3}} 
  \cdot \wt(u_{\hat\lambda}; q,t^q)  \wt(a;q,t)  \wt(b;q,t^{q^{5}}).
\end{aligned}
$$
where we regard $b$ as an element of $\Symm_3$.
\end{example}

\begin{proposition}
\label{form-of-w(q,t)}
Given any $w$ in $\Symm_n$, the weight $\wt(w;q,t)$ for integers $q \geq 2$ 
is a polynomial in $t$ with nonnegative coefficients, taking the following form
$$
\wt(w;q,t) = t^{x} \prod_{i=1}^{\ell(w)} [q]_{t^{q^{y_i}-q^{z_i}}}
$$
for some nonnegative integers $x, y_i, z_i$ with $y_i > z_i$ for all $i$.
Furthermore,
\begin{equation}
\lim_{t \rightarrow 1} \wt(w; q,t) = q^{\ell(w)}
\quad \text{ and } \quad
\lim_{q \rightarrow 1} \wt(w; q,t^{\frac{1}{q-1}}) = t^{\ell(w)}.
\end{equation}

\end{proposition}
\begin{proof}
For all of these assertions, induct on $\ell(w)$, using
the fact that 
$$
\begin{aligned}
\ell(w) & = \ell(u_\lambda) + \ell(a) + \ell(b) \\
        & = k + \ell(u_{\hat\lambda}) + \ell(a) + \ell(b)
\end{aligned}
$$
along with the recursive definition \eqref{first-w-recursion}, 
equation \eqref{factorial-quotient}, and the limits in \eqref{basic-limits}.
\end{proof}

%
%

We first show that this recursively defined $\wt(w;q,t)$ coincides
with $\wt(\lambda,k;q,t)$ when $w=u_\lambda$.

\begin{proposition}
\label{two-part-special-case}
For any partition $\mu$ inside a $k \times (n-k)$ rectangle, one
has $\wt(u_\mu;q,t)=\wt(\mu,k;q,t)$.  Consequently,
$$
\qbin{n}{k}{q,t} = \sum_{u_\mu \in \Symm_n^{(k,n-k)}} \wt( u_\mu; q,t).
$$
\end{proposition}
\begin{proof}
One checks that $\wt(u_\mu;q,t)$ satisfies the same defining
recursion \eqref{lambda-weight-recurrence} as $\wt(\mu;q,t)$.
Temporarily denote
$$
\wt(u_\mu,k;q,t):=\wt(u_\mu; q,t) 
$$
to emphasize the dependence on $k$.  The fact that $w=u_\mu$ lies
in $W^{(k,n-k)}$ implies either $w^{-1}(1)=k+1$ or $1$, 
depending upon whether or not $\mu$ has its first column full of length $k$.
In the former case, one can check that the recursion \eqref{first-w-recursion}
gives
$$
\wt(u_\mu,k;q,t)
= t^{q^k-1} \frac{k!_{q,t^q}}{k!_{q,t}} \cdot \wt(u_{\hat\mu},k;q,t^q) 
$$
and in the latter case that it gives
$$
\wt(u_\mu,k;q,t)
= \wt(u_{\hat\mu},k-1;q,t^q),
$$
as desired.  Thus the equality $\wt(u_\mu,k;q,t)= \wt(\mu,k;q,t)$
follows by induction on $n$.  

The last assertion of the proposition
is simply the restatement of Theorem~\ref{box-interpretation-theorem}.
\end{proof}

In order to generalize Proposition~\ref{two-part-special-case}  
to the $(q,t)$-multinomial, it helps
to have a {\it multinomial $(q,t)$-Pascal relation}.
\begin{proposition} 
\label{multinomial-qtpascal}
For any composition $\alpha=(\alpha_1,\ldots,\alpha_\ell)$ of $n$,
with partial sums $\sigma_s=\sum_{i=1}^s \alpha_s$ (and $\sigma_0:=0$), one has
$$
\qbin{n}{\alpha}{q,t} =
\sum_{i=1}^\ell  
   t^{q^{\sigma_{i-1}}-1} \frac{(\alpha_1,\alpha_2,\ldots,\alpha_{i-1})!_{q,t^q}}
                     {(\alpha_1,\alpha_2,\ldots,\alpha_{i-1})!_{q,t}}
      \qbin{n-1}{\alpha_1,\ldots,\alpha_{i-1},\alpha_i-1,\alpha_{i+1},\ldots,\alpha_n}{q,t^q}
$$
where we recall that
$
\alpha!_{q,t}:=\alpha_1!_{q,t} \, \cdot \, \alpha_2!_{q,t^{q^{\sigma_1}}}  \, \cdot \, 
                \alpha_3!_{q,t^{q^{\sigma_2}}}  \, \cdots  \, \alpha_\ell!_{q,t^{q^{\sigma_{\ell-1}}}}.
$
\end{proposition}
\begin{proof}
Induct on $\ell$, with the base case $\ell=2$ being the first of the
two $(q,t)$-Pascal relations from Proposition~\ref{qtpascal}.  In the inductive
step, write $\alpha=(\alpha_1,\hat\alpha)$ where $\hat\alpha:=(\alpha_2,\ldots,\alpha_\ell)$
is a composition of $n-\alpha_1$.  Beginning with
\eqref{multinomial-in-terms-of-binomials}, start manipulating as follows:
$$
\begin{aligned}
\qbin{n}{\alpha}{q,t} 
& =\qbin{n}{\alpha_1}{q,t}
  \varphi^{\alpha_1} \qbin{n-\alpha_1}{\hat\alpha}{q,t}\\
& =\left( \qbin{n-1}{\alpha_1-1}{q,t^q} 
        + t^{q^{\alpha_1}-1} \frac{\alpha_1!_{q,t^q}}{\alpha_1!_{q,t}} \qbin{n-1}{\alpha_1}{q,t^q}
   \right)
  \varphi^{\alpha_1} \qbin{n-\alpha_1}{\hat\alpha}{q,t}\\
& =\qbin{n-1}{\alpha_1-1}{q,t^q} \varphi^{\alpha_1} \qbin{n-\alpha_1}{\hat\alpha}{q,t} +\\
&    \qquad t^{q^{\alpha_1}-1} \frac{\alpha_1!_{q,t^q}}{\alpha_1!_{q,t}} \qbin{n-1}{\alpha_1}{q,t^q}
       \varphi^{\alpha_1} \qbin{n-\alpha_1}{\hat\alpha}{q,t}
\end{aligned}
$$
The first summand is exactly 
$$
\qbin{n-1}{\alpha_1-1,\hat\alpha}{q,t^q}
$$
which is the $i=1$ term in the proposition.  If one applies the inductive hypothesis
to $\qbin{n-\alpha_1}{\hat\alpha}{q,t}$ in the second summand, one obtains a sum of
$\ell-1$ terms.  When multiplied by the other factors in the second summand,
these give the desired remaining terms $i=2,3,\ldots,\ell$ in the proposition.
\end{proof}

\begin{theorem}
\label{multinomial-as-w(q,t)-sum}
For any composition $\alpha$ of $n$, and $W=\Symm_n$,
$$
\qbin{n}{\alpha}{q,t} = \sum_{w \in W^{\alpha}} \wt(w;q,t).
$$
\end{theorem}
\begin{proof}
Induct on $n$, with the base case $n=1$ being trivial.
If $\alpha=(\alpha_1,\ldots,\alpha_\ell)$ then one can
group the terms in the sum on the right into the subsums
\begin{equation}
\label{subsum}
\sum_{\substack{w \in \Symm_n^{\alpha}:\\ 
        w^{-1}(1) = \sigma_{i-1} + 1}}
       \wt(w;q,t)
\end{equation}
for $i=1,2,\ldots,\ell$.  Introducing the following notations
$$
\begin{aligned}
k &:= \sigma_{i-1}\\
\alpha' &:= (\alpha_1,\alpha_2,\ldots,\alpha_{i-1})\\
\alpha''&:= (\alpha_i-1,\alpha_{i+1},\alpha_{i+2},\ldots,\alpha_{\ell})\\
\end{aligned}
$$
we wish to show that the subsum \eqref{subsum} equals
the following term from the right side of Proposition~\ref{multinomial-qtpascal}:
\begin{equation}
\label{desired-Pascal-term}
t^{q^k-1} \frac{\alpha'!_{q,t^q}}{\alpha'!_{q,t}} 
      \qbin{n-1}{\alpha',\alpha''}{q,t^q}.
\end{equation}

Note that when $w^{-1}(1) = k + 1$,
the recursive definition of $\wt(w;q,t)$ says
$$
\wt(w;q,t) = t^{q^k-1} \frac{k!_{q,t^q}}{k!_{q,t}} \cdot \wt(u;q,t^q) \wt(a; q,t) \wt(b; q,t^{q^{k+1}})
$$
where $u \in \Symm_{n-1}^{(k,n-1-k)}, a \in  \Symm_k^{\alpha'}, b\in  \Symm_{n-1-k}^{\alpha''}.$
Thus one can rewrite \eqref{subsum} as
$$
\begin{aligned}
&t^{q^k-1} \frac{k!_{q,t^q}}{k!_{q,t}}
\sum_{u \in \Symm_{n-1}^{(k,n-1-k)}}  \wt(u;q,t^q)
\sum_{a \in  \Symm_n^{\alpha'} } \wt(a;q,t)
\sum_{b\in  \Symm_n^{\alpha''} } \wt(b;q,t^{q^{k+1}}) \\
&=t^{q^k-1} \frac{k!_{q,t^q}}{k!_{q,t}}
\qbin{n-1}{k,n-1-k}{q,t^q}
\qbin{k}{\alpha'}{q,t}
\qbin{n-1-k}{\alpha''}{q,t^{q^{k+1}}}\\
&=t^{q^k-1} \frac{(n-1)!_{q,t^q}} {\alpha'!_{q,t} \alpha''!_{q,t^{q^{k+1}}}} \\
&=t^{q^k-1} \frac{\alpha'!_{q,t^q}}{\alpha'!_{q,t}} \qbin{n-1}{\alpha',\alpha''}{q,t^q}
\end{aligned}
$$
in which the first equality replaced all three sums;
Proposition~\ref{two-part-special-case} was used to replace the first sum,
while the inductive hypothesis was used to replace the second and third sums.
\end{proof}

\section{Ribbon numbers and descent classes}
\label{Macmahon-section}

Recall that the minimum-length coset representatives $W^\alpha$ for $W/W_\alpha$ are
the permutations $w$ in $W=\Symm_n$ whose descent composition
$\beta(w)$ is refined by $\alpha$.  The set of permutations $w$ for which
$\beta(w)=\alpha$ is sometimes called a {\it descent class}.
We define in terms of these classes
the {\it ribbon}, {\it $q$-ribbon}, and {\it $(q,t)$-ribbon numbers} 
for a composition $\alpha$ of $n$:
\begin{equation}
\label{ribbon-number-definitions}
\begin{aligned}
r_\alpha   &:=| \{w \in W:\alpha = \beta(w)\} |, \\
r_\alpha(q)&:= \sum_{\substack{w \in W:\\ \alpha = \beta(w)}} q^{\ell(w)},\\
r_\alpha(q,t)&:= \sum_{\substack{w \in W:\\ \alpha = \beta(w)}} \wt(w; q,t).\\
\end{aligned}
\end{equation}
Recall that the partial order by refinement on the $2^{n-1}$ compositions $\alpha$ of
$n$ is isomorphic to the partial order by inclusion of their subsets of
partial sums 
$$
\{\alpha_1,\alpha_1+\alpha_2,\ldots,\alpha_1+\cdots+\alpha_{\ell-1}\}.
$$
From \eqref{multinomial-interps} and Theorem~\ref{multinomial-as-w(q,t)-sum} 
it should be clear that there is an inclusion-exclusion relation between
these three kinds of the ribbon numbers and three kinds of
multinomials (ordinary, $q$-, and $(q,t)$-multinomials).

However, it turns out that the inclusion-exclusion formula for the ribbons
collates into a determinantal formula involving factorials.  This determinant
for ribbon numbers goes back to MacMahon, for $q$-ribbon numbers 
to Stanley (see \cite[Examples 2.2.5]{Stanley-EC1}), and for $(q,t)$-ribbon numbers 
is new, although all three are proven in the same way; 
see Stanley \cite[Examples 2.2.4,2.2.5]{Stanley-EC1}).

\begin{proposition}
\label{Macmahon-determinant}
For any composition $\alpha=(\alpha_1,\ldots,\alpha_\ell)$ of $n$, with
partial sums $\sigma_i:=\sum_{j=1}^i \alpha_j$, one has
$$
\begin{aligned}
r_\alpha &= \sum_{\beta \text{ refined by } \alpha} (-1)^{\ell(\alpha)-\ell(\beta)} \binom{n}{\beta}
 = n!    \det\left(  
                    \frac{1}{(\sigma_j-\sigma_{i-1})!}
                  \right)_{i,j=1}^{\ell(\alpha)}\\
r_\alpha(q) &= \sum_{\beta \text{ refined by } \alpha} (-1)^{\ell(\alpha)-\ell(\beta)} \qbin{n}{\beta}{q}
 = [n]!_q    \det\left(  
                    \frac{1}{[\sigma_j-\sigma_{i-1}]!_{q}}
                  \right)_{i,j=1}^{\ell(\alpha)}\\
r_\alpha(q,t) &= \sum_{\beta \text{ refined by } \alpha} (-1)^{\ell(\alpha)-\ell(\beta)} \qbin{n}{\beta}{q,t}
 = n!_{q,t}    \det\left(  
                    \varphi ^{\sigma_{i-1}} \frac{1}{(\sigma_j-\sigma_{i-1})!_{q,t}}
                  \right)_{i,j=1}^{\ell(\alpha)}\\
\end{aligned}
$$
where $[m]!_q:=1(1+q)(1+q+q^2) \cdots (1+q+q^2+\cdots+q^{m-1})$.
\end{proposition}

By the definition \eqref{ribbon-number-definitions}, it
is clear that $r_\alpha$ is nonnegative, that $r_\alpha(q)$ is a polynomial in
$q$ with nonnegative coefficients, and that for integers $q \geq 2$ one will
have $r_\alpha(q,t)$ a polynomial in $t$ with nonnegative
coefficients.  It should also be clear that
$$
\begin{aligned}
\lim_{q \rightarrow 1} r_\alpha(q) & =r_\alpha \\
\lim_{t \rightarrow 1} r_\alpha(q,t) & = r_\alpha(q) \\
\lim_{q \rightarrow 1} r_\alpha(q,t^{\frac{1}{q-1}}) &= r_\alpha(t).
\end{aligned}
$$
Our goal in the next section will be to interpret these three ribbon numbers homologically.

\section{Homological interpretation of ribbon numbers}
\label{building-section}

The ribbon number $r_\alpha$ has a well-known interpretation 
as the rank of the only non-vanishing homology group in the $\alpha$-rank-selected subcomplex 
$\Delta(W,S)_\alpha$ of the {\it Coxeter complex} $\Delta(W,S)$ for $W=\Symm_n$.
For prime powers $q$, a result of Bj\"orner \cite[Theorem 4.1]{Bjorner} analogously shows that
$r_\alpha(q)$ is the rank of the homology in the $\alpha$-rank-selected subcomplex 
of the {\it Tits building} $\Delta(G,B)$ for $G=GL_n(\FF_q)$.

Here we use Bj\"orner's results to give, in parallel, Hilbert series interpretations
for $r_\alpha(t), r_\alpha(q,t)$.  These interpretations will be
related to graded modules of $\Hom$ spaces between the homology representations on $\Delta(W,S)_\alpha$
or $\Delta(G,B)_\alpha$ and appropriate polynomial rings.
This generalizes work of Kuhn and Mitchell \cite{KuhnMitchell}, 
who dealt with the case where $\alpha=(1,1,\ldots,1)=:1^n$,
in order to determine the (graded) composition multiplicities of the {\it Steinberg character}
of $G$ within the polynomial ring $\FF_q[\xx]$.

\begin{definition}
Let $W:=\Symm_n$ and $G:=GL_n(\FF_q)$.
Given a composition $\alpha$ of $n$, define the virtual sum of induced $\ZZ W$-modules
\begin{equation}
\label{W-virtual-modules}
\chi^\alpha := \sum_{ \beta \text{ refined by }\alpha } (-1)^{\ell(\alpha)-\ell(\beta)} 1_{W_\beta}^{W}
\end{equation}
and $\FF_q G$-modules
\begin{equation}
\label{G-virtual-modules}
\chi^\alpha_q := \sum_{ \beta \text{ refined by }\alpha } (-1)^{\ell(\alpha)-\ell(\beta)} 1_{P_\beta}^{G}.
\end{equation}
\end{definition}
These virtual modules have been considered by Bj\"orner, Bromwich, Curtis,
Mathas, Smith, Solomon, Surowski, and others; see \cite{Bjorner}, \cite{Mathas} and
\cite{Smith} for some of the relevant references.  
In the special case where $\alpha=1^n$ is a single column with 
$n$ cells, $\chi^\alpha$ is the {\it sign} representation of $W$, and
$\chi^\alpha_q$ is the {\it Steinberg representation} of $G$.

For any composition $\alpha$ of $n$, these virtual modules $\chi^\alpha$ and
$\chi^\alpha_q$ turn out to be {\it genuine} $\ZZ W$ and $\ZZ G$-modules.  
They can be defined over the integers because they are the 
representations on the {\it top} homology of the (shellable) simplicial complexes
$\Delta(W,S)_\alpha$ and $\Delta(G,B)_\alpha$,
which are the {\it rank-selection} (or {\it type-selection}) of 
the Tits building $\Delta(G,B)$ to the rank set given by the partial sums 
$\{ \sigma_s \}_{s=1,\ldots,\ell-1}$; see \cite[\S 4]{Bjorner}.  Note that top-dimensional
homology groups are always free as $\ZZ$-modules because they are the group of top-dimensional 
{\it cycles}; there are no boundaries to mod out.

In what follows, we will make several arguments about why certain algebraic complexes
$$
\cdots \rightarrow C_{i+1} \overset{d_{i+1}}{\rightarrow} C_i \overset{d_i}{\rightarrow} C_{i-1} \rightarrow \cdots
$$
are not only acyclic, but actually{\it chain-contractible}, that is, 
there exist maps backward $C_{i+1} \overset{D_i}{\leftarrow} C_i$
for each $i$ with the property that
$$
D_{i-1} d_i + d_{i+1} D_i = 1_{C_i}.
$$
We will use repeatedly the following key fact.
\begin{proposition}
\label{additive-functor-prop}
If one applies an additive functor to a chain-contractible complex, the
result remains chain-contractible.
\end{proposition}
\begin{proof}
If the complex is called $(\mathcal C,d_*)$ and the functor called $F$, then
the maps $F(D_i)$ provide a chain-contraction for $(F({\mathcal C}), F(d_*))$:
additivity and functoriality imply
$$
F(D_{i-1}) F(d_i) + F(d_{i+1}) F(D_i) = 1_{F(C_i)}
$$
if $F$ is covariant, and a similar statement if $F$ is contravariant..
\end{proof}

The following key fact was proven by Kuhn and Mitchell for $\alpha=1^n$;  we
simply repeat their proof for general $\alpha$.

\begin{theorem}
\label{KM}
Given a composition $\alpha$ of $n$, the simplicial chain complex for the type-selection $\Delta(W,S)_\alpha$ or
$\Delta(G,B)_\alpha$ gives rise to chain-contractible complexes of $\ZZ W$ or $\FF_q G$-modules
\begin{equation}
\label{chain-contractible-complexes}
\begin{aligned}
&0 \rightarrow \chi^{\alpha} \rightarrow {\mathcal {C}} \text{ or }\\
&0 \rightarrow \chi^{\alpha}_q \rightarrow {\mathcal {C}}\\
\end{aligned}
\end{equation}
where the typical term in ${\mathcal C}$ takes the form
$$
\begin{aligned}
&\bigoplus_{\substack{\beta \text{ refined by }\alpha:\\ \ell(\beta)=k}} 1_{W_\beta}^{W} \text{ or }\\
&\bigoplus_{\substack{\beta \text{ refined by }\alpha:\\ \ell(\beta)=k}} 1_{P_\beta}^{G}.
\end{aligned}
$$
\end{theorem}

\begin{proof}
We give the proof for the case of the Tits building
$\Delta(G,B)$; the ``$q=1$ case'' for $\Delta(W,S)$ is even easier.

First note that $\chi^\alpha_q$ includes in the first (top) chain group as the kernel
of the top boundary map, setting up the complex of $\ZZ G$-modules in 
\eqref{chain-contractible-complexes}.  It remains to prove that it is
chain-contractible after tensoring with $\FF_q$.

Bj\"orner \cite{Bjorner}, Kuhn and Mitchell \cite{KuhnMitchell}, and Smith \cite{Smith}
have observed that the shelling order which one uses
for the Tits building (or any of its rank-selections) can actually be chosen $B$-equivariant:
one can shell the facets $bwP_\alpha$ in any order that respects the 
ordering by length of the minimal
coset representative $w \in W/W_\alpha$, and the $B$-action never alters this representative
$w$.  This means that the resulting chain-contraction maps 
can be chosen as $\ZZ B$-module maps.

Since $[G:B]$ is coprime to the prime $p$ (= the characteristic of $\FF_q$),
if one tensors the coefficients with the localization $\mathbb{Z}_{(p)}$
at the prime $p$ (i.e. inverting all elements of $\mathbb{Z}$ coprime to $p$), one can start with
these $\mathbb{Z}_{(p)} B$-module maps, and average them over the cosets $G/B$ to obtain
$\mathbb{Z}_{(p)}G$-module maps that still give a chain-contraction.  

Lastly, one can tensor the coefficients
with $\mathbb{F}_q$ and obtain the desired $\mathbb{F}_q G$-module chain-contraction.
\end{proof}

Given an $\FF_q G$-module $\psi$, one can
regard the $\FF_q$-vector space $\Hom_{\FF_q G}(\psi,\FF_q[\xx])$ as an 
$\FF_q[\xx]^{G}$-module:  
given $f$ in $\FF_q[\xx]^G$, and a $G$-equivariant map $h: \psi \rightarrow \FF_q[\xx]$,
the map $fh$ that sends $u \in \psi$ to $f \cdot h(u)$ is also $G$-equivariant.

We come to the main result of this section, whose assertion for $\Symm_n$-representations
is known in characteristic zero;  see the extended Remark~\ref{cell-repn-remark} below.

\begin{theorem}
\label{hilbdet}
Given a composition $\alpha$ of $n$,  
the $\ZZ[\xx]^W$-module 
$$
M:= \Hom_{\ZZ W}(\chi^\alpha, \ZZ[\xx])
$$
is free over $\ZZ[\xx]^W$, with
$$
\Hilb(M/\ZZ[\xx]^W_+ M,t) = r_\alpha(t).
$$
Analogously, for $q$ a prime power, the $\FF_q[\xx]^G$-module
$$
M:= \Hom_{\FF_q G}(\chi^\alpha_q, \FF_q[\xx])
$$
is free over $\FF_q[\xx]^G$, with
$$
\Hilb(M/\FF_q[\xx]^G_+ M,t) = r_\alpha(q,t).
$$
\end{theorem}

\begin{proof}
As with the previous theorem, we give the proof only for the assertions about $G$;  the proof
for the assertions about $W$ are analogous and easier.

Start with the chain-contractible $\FF_q GL_n$-complex from Theorem 
\ref{KM}.   Applying the functor $\Hom_{\FF_q GL_n}(-,\FF_q[\xx])$ to this, one obtains (via Proposition~\ref{additive-functor-prop}) a 
chain-contractible complex of $\FF_q[\xx]^G$-modules
that looks like
$$
{\mathcal C}' \rightarrow M \rightarrow 0
$$
and where the typical term in ${\mathcal C}'$ is a direct sum of terms of the form
$$
\Hom_{\FF_q G}(1_{P_\beta}^G,\FF_q[\xx]) \cong \FF_q[\xx]^{P_\beta}.
$$
Since every ring $\FF_q[\xx]^{P_\beta}$ is a free $\FF_q[\xx]^G$-module by Corollary \ref{free},
this is actually a free $\FF_q[\xx]^G$-resolution of
$M$.  Thus it can be used to compute $\Tor^{\FF_q[\xx]^G}(M,\FF_q)$:  tensoring ${\mathcal C}'$
over $\FF_q[\xx]^G$ with $\FF_q$ gives (via Proposition~\ref{additive-functor-prop}) a 
chain-contractible complex ${\mathcal C}''$ of $\FF_q$-vector spaces, whose homology
computes this $\Tor$.  But since the complex  
${\mathcal C}''$ is chain-contractible, $\Tor^{\FF_q[\xx]^G}_i(M,\FF_q)$ vanishes
for $i > 0$, that is, $M$ is a free $\FF_q[\xx]^G$-module, giving the first assertion of the theorem.

For the second assertion, note that the resolution  ${\mathcal C}'$ of the $\FF_q[\xx]^G$-module $M$ shows
$$
\Hilb(M,t) = \sum_{\beta \text{ refined by } \alpha} (-1)^{\ell(\alpha)-\ell(\beta)} \Hilb(\FF_q[\xx]^{P_\beta},t).
$$
Since $M$ and every one of the $\FF_q[\xx]^{P_\beta}$'s are all free as $\FF_q[\xx]^G$-modules
by Corollary \ref{free}, one can divide both sides by $\Hilb(\FF_q[\xx]^G,t)$ to obtain
$$
\begin{aligned}
\Hilb(M/\FF_q[\xx]^G_+ M,t)&= \sum_{\beta \text{ refined by } \alpha} 
                                     (-1)^{\ell(\alpha)-\ell(\beta)} 
                          \frac{\Hilb(\FF_q[\xx]^{P_\beta},t)}{\Hilb(\FF_q[\xx]^G,t)}\\
                      &= \sum_{\beta \text{ refined by } \alpha} (-1)^{\ell(\alpha)-\ell(\beta)} 
                            \qbin{n}{\beta}{q,t}\\
                      & = r_\alpha(q,t). \qed 
\end{aligned}
$$
\end{proof}

\begin{remark}
\label{cell-repn-remark}
We sketch here how the assertion in Theorem~\ref{hilbdet} for $W=\Symm_n$ follows from
known results in the literature, when one considers $\CC W$-modules rather than $\ZZ W$-modules;
see Roichman \cite{Roichman} for generalizations and more recent viewpoints on some of these results.

It is known from work of Hochster and Eagon (see \cite[Theorem 3.10]{Stanley-invariants})
that for {\it any} $\CC W$-module $\chi$, the $\Hom$-space
$
M^\chi:=\Hom_{\CC W}(\chi, \CC[\xx])
$
is always free as a $\CC[\xx]^W$-module.  One can compute its Hilbert series via
a {\it Molien series} calculation \cite[Theorem 2.1]{Stanley-invariants} as 
\begin{equation}
\label{Molien-calculation}
\begin{aligned}
\Hilb( M^\chi,t) &= \frac{1}{n!} \sum_{w \in W=\Symm_n} \frac{ \chi(w) }{\det(1 - t w)} \\
&= \frac{1}{n!} \sum_{w \in \Symm_n} \chi(w) \cdot p_{\lambda(w)}(1,t,t^2,\ldots) \\
&= s_\chi(1,t,t^2,\ldots).
\end{aligned}
\end{equation}
Here $\lambda(w)$ denotes the partition of $n$ that gives $w$'s cycle type,
$p_\lambda(x_1,x_2,\ldots)$ is the {\it power sum} symmetric function
corresponding to $\lambda$, and $s_\chi(x_1,x_2,\ldots)$ is the symmetric function
which is the image of the character $\chi$ under the 
{\it Frobenius characteristic map} from
$\Symm_n$-characters to symmetric functions.  
Thus one has
$$
\begin{aligned}
\Hilb( M^\chi/\CC[\xx]^W_+ M^\chi,t) 
 &=  \frac{\Hilb( M^\chi,t) }{\Hilb( \CC[\xx]^W,t) } \\
 &= (t;t)_n \,\,  s_\chi(1,t,t^2,\ldots).
\end{aligned}
$$
When $\chi$ is a {\it skew-character} $\chi^{\lambda/\mu}$ of $\Symm_n$, 
then $s_\chi=s_{\lambda/\mu}$ is a skew Schur function, and
one has the explicit formula \cite[Proposition 7.19.11]{Stanley-EC2}
\begin{equation}
\label{skew-fake-degree}
(t;t)_n \,\, s_{\lambda/\mu}(1,t,t^2,\ldots) 
 = f^{\lambda/\mu}(t) := \sum_{Q} t^{\maj(Q)}
\end{equation}
where $Q$ runs over standard Young tableaux of shape $\lambda/\mu$, and
$\maj(Q)$ is the sum of the entries $i$ in the {\it descent set} defined by
$$
\Des(Q):=\{i \in \{1,2,\ldots,n-1\}: i+1 \text{ appears in a lower row of }Q\text{ than }i\}.
$$
Given a composition $\alpha=(\alpha_1,\ldots,\alpha_\ell)$, the 
$\CC W$-module $\chi^\alpha$ on the top homology of
the subcomplex $\Delta(W,S)_\alpha$ turns out to be the skew-character $\chi^{\lambda/\mu}$
for the ribbon skew shape $\lambda/\mu$ having $\alpha_i$ cells in its $i^{th}$ lowest row: 
Solomon \cite{Solomon} used the Hopf trace formula to express the homology representation
$\chi^\alpha$ as the virtual sum in \eqref{W-virtual-modules}, and this can then be re-intepreted
as the {\it Jacobi-Trudi} formula for the skew-character of this ribbon skew shape.  
Consequently, if $M:=M^{\chi^\alpha}$ then
$$
\Hilb(M/\CC[\xx]^W_+ M,t)
 = \sum_{Q: \lambda(Q)=\alpha} t^{\maj(Q)} 
 = \sum_{\substack{w \in \Symm_n : \\ \beta(w)=\alpha}} t^{\maj(w)} 
  = \sum_{\substack{w \in \Symm_n :\\ \beta(w)=\alpha}} t^{\ell(w)}
  = r_\alpha(t).
$$
Here the second equality uses a well-known bijection that reads a standard tableaux $Q$
filling the ribbon shape and associates to it a permutation $w$ in $W$ having
descent composition $\beta(w)=\alpha$, while
the third equality is a well-known result of MacMahon (see e.g., \cite{FoataSchutzenberger}).
\end{remark}

\section{The coincidence in the case of hooks}
\label{hook-section}

As mentioned in the Introduction,
there is an important coincidence that occurs in the special case of
the principal specialization of $\MacSchurZ_\lambda$ when $\lambda$ is
a {\it hook shape} $(m,1^k)$, leading to a relation with the $(q,t)$-ribbon
number for the reverse hook composition $\alpha=(1^k,m)$.  

We begin with a simplification in the product formula for
the principal specialization when $\lambda$ is a hook.

\begin{proposition} 
\label{qthook}
For $n \geq k$,
$$
\begin{aligned}
\MacSchurZ_{(m,1^k)}(1,t,\cdots, t^n)
& = \qbin{n+m}{n-k}{q,t} \varphi^{n-k}
\prod_{i=1}^k \frac{t^{q^{m+k}} - t^{q^i}}{t^{q^i} - t} \\
&= \qbin{n+m}{n-k}{q,t} \varphi^{n-k} \MacSchurZ_{(m,1^k)}(1,t,\cdots, t^k).
\end{aligned}
$$
\end{proposition}
\begin{proof}
The second equation is a consequence of the first.  The first equation is a
straightforward consequence of \eqref{bialternant-product-formula}, 
deduced similarly to the proof 
of Theorem~\ref{qthomo} and equation~\eqref{qtelem}.
\end{proof}

\noindent
This implies the following relation, that we will use below for an induction.
\begin{corollary}
\label{hookrec}
$$
\begin{aligned}
& \MacSchurZ_{(m,1^k)}(1,t,t^2,\ldots,t^n)+ \MacSchurZ_{(m+1,1^{k-1})}(1,t,t^2,\ldots,t^{n-1})\\
&\qquad = \HZ_m(1,t,t^2,\ldots,t^n) \cdot \EZ_k(1,t,t^2,\ldots,t^{n-1})
\end{aligned}
$$
\end{corollary}
\begin{proof}
Apply Proposition~\ref{qthook} to the left side:
$$
\begin{aligned}
&\MacSchurZ_{(m,1^k)}(1,t,t^2,\ldots,t^n)+ \MacSchurZ_{(m+1,1^{k-1})}(1,t,t^2,\ldots,t^{n-1})\\
&= \qbin{n+m}{n-k}{q,t} \varphi^{n-k} \prod_{i=1}^k \frac{t^{q^{m+k}} - t^{q^i}} {t^{q^i} - t} 
+ \qbin{n+m}{n-k}{q,t} \varphi^{n-k} \prod_{i=1}^{k-1} \frac{t^{q^{m+k}} - t^{q^i}} {t^{q^i} - t} \\
&= \qbin{n+m}{n-k}{q,t} \varphi^{n-k} \left( \left( \frac{t^{q^{m+k}} - t^{q^k}}{t^{q^k} - t} + 1 \right) 
                            \prod_{i=1}^{k-1} \frac{t^{q^{m+k}} - t^{q^i}} {t^{q^i} - t} \right) \\
&= \qbin{n+m}{n-k}{q,t} \varphi^{n-k} \prod_{i=1}^k \frac{t^{q^{m+k}} - t^{q^{i-1}}} {t^{q^i} - t} \\
&= \qbin{n+m}{n-k}{q,t} \varphi^{n-k} \left( \qbin{m+k}{k}{q,t} 
        \prod_{i=1}^k \frac{t^{q^k} - t^{q^{i-1}}} {t^{q^i} - t} \right) \\
&= \qbin{n+m}{n}{q,t} \cdot \qbin{n}{n-k}{q,t} \varphi^{n-k} \prod_{i=1}^k \frac{t^{q^k} - t^{q^{i-1}}} {t^{q^i} - t}\\
&= \HZ_m(1,t,t^2,\ldots,t^n) \cdot \EZ_k(1,t,t^2,\ldots,t^{n-1}).
\end{aligned}
$$
The last equality used \eqref{qtelem}.
\end{proof}

\begin{theorem}
\label{hookhilb}
$$
\MacSchurZ_{(m,1^k)}(1,t,t^2,\ldots,t^n) = \qbin{m+n}{n-k}{q,t} r_{(1^k,m)}(q,t^{q^{n-k}}).
$$                                                                               
\end{theorem}
\begin{proof}
By the second equation in Proposition~\ref{qthook}, it suffices to prove this
in the case $n=k$, that is,
\begin{equation}
\label{k-equal-n-hook-coincidence}
\MacSchurZ_{(m,1^k)}(1,t,t^2,\ldots,t^k) = r_{(1^k,m)}(q,t).
\end{equation}                    
Let $LHS(m,k), RHS(m,k)$ denote the left, right sides in \eqref{k-equal-n-hook-coincidence}.
We will show they are equal by induction on $k$; in the base case $k=0$
both are easily seen to equal $1$.

For the inductive step, note that Corollary \ref{hookrec} 
gives the following recurrence on $k$ for $LHS(m,k)$:
$$
LHS(m,k) = - LHS(m+1,k-1) + \qbin{m+k}{k}{q,t} LHS(1,k-1).
$$
To show $RHS(m,k)$ satisfies the same recurrence,
start with the summation expression for $RHS(m,k)$ given in Theorem \ref{hilbdet}:
$$
RHS(m,k) = \sum_{\beta \text{ refined by } (1^k, m)} (-1)^{k+1-\ell(\beta)} \qbin{m+k}{\beta}{q,t}.
$$
Classify the terms indexed by $\beta$ in this sum 
according to whether the composition $\beta$ ends in a last part strictly 
larger than $m$, or equal to $m$.  The former terms correspond
to compositions $\beta \text{ refined by } (1^{k-1}, m+1)$, and their sum gives rise to the 
desired first term $-RHS(m+1,k-1)$ in the
recurrence.  The latter terms correspond, by removing the last part of $\beta$ of size $m$, 
to compositions $\hat\beta$ refined by $1^k$, and
their sum is 
$$
\begin{aligned}
&\sum_{ \substack{\beta \text{ refined by }(1^k, m) \\ \text{ ending in }m} } 
            (-1)^{k+1-\ell(\beta)} \qbin{m+k}{\beta}{q,t} \\
 &\quad =  \qbin{m+k}{k}{q,t} \cdot \sum_{ \hat\beta \text{ refined by }1^k } (-1)^{k-\ell(\hat\beta)} 
               \qbin{k}{\hat\beta}{q,t} \\
 &\quad =\qbin{m+k}{k}{q,t} \cdot RHS(1,k-1),
\end{aligned}
$$
that is, the desired second term in the recurrence.
\end{proof}

\begin{remark}
Note that equation \eqref{k-equal-n-hook-coincidence}, together with Theorem~\ref{hilbdet},
gives the principal specialization $\MacSchurZ_{(m,1^k)}(1,t,t^2,\ldots,t^n)$ an algebraic
interpretation in the special case $n=k$.  However, this generalizes in a straightforward
fashion when $n \geq k$, as the same methods that prove Theorem~\ref{hilbdet}
can be used to prove the following.

Let $A:=\FF_q[x_1,\ldots,x_{m+n}]$ with its usual action of $G_{m+n}:=GL_{m+n}(\FF_q)$.
Given $\alpha$ a composition of $n$, consider the induced $\FF_q G_{m+n}$-module 
$$
\chi^{\alpha,m+n}_q := \Ind_{P_{(m,n)}}^{G_{m+n}} \chi^\alpha_q
$$
where one considers the homology representation $\chi^\alpha_q$ for
$GL_n(\FF_q)$ as also a representation for the parabolic subgroup $P_{(m,n)}$,
via the homomorphism $P_{(m,n)} \rightarrow GL_n(\FF_q)$ that ignores all
but the lower right $n \times n$ submatrix.

\begin{theorem}
\label{hilbdet-general}
In the above situation, the $A^{G_{m+n}}$-module
$$
M:= \Hom_{\FF_q G_{m+n}}(\chi^{\alpha,m+n}_q, A)
$$
is free over $A^{G_{m+n}}$, with
$$
\Hilb(M/A^{G_{m+n}}_+ M,t) 
= \qbin{m+n}{n-k}{q,t} r_{\alpha}(q,t^{q^{n-k}}).
$$
\end{theorem}

\noindent
When $\alpha=(1^k,m)$, the right side above is
$\MacSchurZ_{(m,1^k)}(1,t,t^2,\ldots,t^n)$, by Theorem~\ref{hookhilb}.

\end{remark}

\section{Questions}
\label{questions-section}

%
%
%
%
\subsection{Bases for the quotient rings and Schubert calculus}
\label{quotient-basis-question}

Is there a simple explicit basis one can write down for $\FF_q[\xx]^{P_\alpha}/(\FF_q[\xx]_+^G)$?
By analogy to Schubert polynomial theory, it would be desirable to have
a basis when $\alpha=1^n$, containing the basis for 
any other $\alpha$ as a subset.

\subsection{The meaning of the principal specializations}
\label{principal-specialization-question}

What is the algebraic (representation-theoretic, Hilbert series?) meaning of 
$\MacSchurZ_{\lambda/\mu}(1,t,\ldots,t^n)$, or perhaps just the non-skew special case where
$\mu=\varnothing$? Is there an algebraic meaning to 
the elements $\MacSchurZ_{\lambda/\mu}(x_1,x_2,\ldots,x_n)$ lying in $\Qq[\xx]$?

\subsection{A $(q,t)$ version of the fake degrees?}
\label{qt-fake-degree-question}

The sum appearing in \eqref{skew-fake-degree}
is a skew generalization of the usual {\it fake-degree} polynomial
\begin{equation}
\label{q-hook-formula}
f^{\lambda}(t) = \sum_{Q} t^{\maj(Q)} 
= q^{b(\lambda)} \frac{[n]!_q}{\prod_{x} [h(x)]_q}
\end{equation}
where $b(\lambda)=\sum_{i \geq 1} \binom{\lambda'_i}{2}$,
and $h(x)$ is the {\it hook length} of the cell $x$ of $\lambda$;
see \cite[Corollary 7.12.5]{Stanley-EC2}.  The fake degree polynomial
$f^{\lambda}(q)$ has a different meaning when $q$ is a prime power, giving
the dimension of the complex {\it unipotent representations}
of $GL_n(\FF_q)$ considered originally by Steinberg \cite{Steinberg}; see \cite{James, Olsson}.

Is there a $(q,t)$-fake degree polynomial $f^{\lambda/\mu}(q,t) \in \Qq(\ttt)$
having any or all of the following properties (a)-(f)?:

\begin{enumerate}
\item[(a)] $\lim_{t \rightarrow 1} f^{\lambda/\mu}(q,t) = f^{\lambda/\mu}(q)$.
\item[(b)] $\lim_{q \rightarrow 1} f^{\lambda/\mu}(q,t^{\frac{1}{q-1}}) = f^{\lambda/\mu}(t)$.
\item[(c)] Better yet, a summation formula generalizing \eqref{skew-fake-degree} of the form
$$
f^{\lambda/\mu}(q,t) = \sum_{Q} \wt(Q;q,t)
$$
where $Q$ runs over all standard Young tableau of shape $\lambda/\mu$.
Here $\wt(Q;q,t)$ should be an element of $\Qq(\ttt)$ with a product formula that
shows it is a polynomial in $t$ with nonnegative coefficients
for integers $q \geq 2$, and that
$$
\lim_{t \rightarrow 1} \wt(Q;q,t) = q^{\maj(Q)} \text{ and }
\lim_{q \rightarrow 1} \wt(Q;q,t^{\frac{1}{q-1}}) = t^{\maj(Q)}.
$$
(Note that property (c) would imply properties (a), (b)).
\item[(d)] A Hilbert series interpretation $q$-analogous to \eqref{skew-fake-degree} of the form
\begin{equation}
\label{fake-degree-as-hilb-numerator}
\begin{aligned}
\Hilb( M , t)&=\frac{ f^{\lambda/\mu}(q,t) }{(1-t^{q^n-1})(1-t^{q^n-q}) \cdots(1-t^{q^n-q^{n-1}}) } \\
&= f^{\lambda/\mu}(q,t) \cdot \Hilb( K[\xx]^G, t)  
\end{aligned}
\end{equation}
Here $M$ should be a graded $K[\xx]^G$-module (not necessarily free),
where $K$ is some extension field of $\FF_q$ and $G=GL_n(\FF_q)$ acts in
the usual way on $K[\xx]:=K[x_1,\ldots,x_n]$.
Equivalently this would mean that 
$$
f^{\lambda/\mu}(q,t) = \sum_{ i \geq 0 }(-1)^i \Hilb( \Tor^{K[\xx]^G}_i(M,K), t).
$$
\item[(e)] A reinterpretation of the power series in \eqref{fake-degree-as-hilb-numerator}
as a principal specialization of some symmetric function in an infinite variable set,
generalizing \eqref{skew-fake-degree}.

\item[(f)]  When $\mu = \varnothing$, a product {\it $(q,t)$-hook length formula} for 
$f^{\lambda}(q,t)$ generalizing \eqref{q-hook-formula}.
\end{enumerate}

When $\lambda=(m,1^k)$ is a hook shape, one can 
define $f^{\lambda}(q,t):=r_\alpha(q,t)$ for $\alpha=(1^k,m)$.  Then our previous
results on $r_\alpha(q,t)$ can be loosely re-interpreted as verifying properties (a),(b),(c),(d),(f)
(but not, as far as we know, (e)). 
Here the desired module $M$ is $\Hom_{\FF_q G}(\chi^\alpha_q,\FF_q[\xx])$, where
$\chi^\alpha_q$ was the homology representation on the $\alpha$-type-selected subcomplex of the
Tits building $\Delta(G,B)$ considered in Section~\ref{building-section}.  
It is known that this homology representation $\chi^\alpha_q$ is
an integral lift of the complex unipotent character considered by Steinberg in this case.

\subsection{An equicharacteristic $q$-Specht module for $\FF_q GL_n(\FF_q)$?}
\label{q-Specht-section}

The discussion in Section~\ref{qt-fake-degree-question} and Remark~\ref{cell-repn-remark}
perhaps suggests the existence of a generalization of the 
$\FF_q G$-module $\chi^\alpha_q$ for $G=GL_n(\FF_q)$
from ribbon shapes $\alpha$ to {\it all} skew shapes $\lambda/\mu$.

\begin{question}
\label{overarching-question}
Can one find a field extension $K$ of $\FF_q$
and a $K G$-module $\chi^{\lambda/\mu}_q$ 
which is a $q$-analogue of the skew Specht module $\chi^{\lambda/\mu}$ for $\Symm_n$,
generalizing the homology representation $\chi_q^\alpha = \chi^{\lambda/\mu}_q$ when 
$\alpha$ is a ribbon skew shape as in Remark~\ref{cell-repn-remark},
and playing the following three roles?
\end{question}

\vskip .1 in
\noindent
{\sf Role 1.}
Let $\lambda/\mu$ be an arbitrary skew shape.
It is known from work of James and Peel \cite{JamesPeel}
that any skew Specht modules $\chi^{\lambda/\mu}$ for $W=\Symm_n$ 
has a characteristic-free {\it Specht series}, that is,
a filtration in which each factor is isomorphic to a {\it non-skew}
Specht module $\chi^{\nu}$, and where the number of factors
isomorphic to $\chi^{\nu}$ is equal to the Littlewood-Richardson
number $c^{\nu}_{\lambda/\mu} = c^{\lambda}_{\mu,\nu}$.

\begin{question} 
Does the hypothesized $q$-skew Specht $KG$-module $\chi^{\lambda/\mu}_q$ from 
Question~\ref{overarching-question} have a $K G$-module filtration 
in which each factor is isomorphic to one of the non-skew 
$q$-Specht modules $\chi^{\nu}_q$, and where the number of factors
isomorphic to $\chi^{\nu}_q$ is equal to the Littlewood-Richardson
number $c^{\nu}_{\lambda/\mu}$?
\end{question}

%
%

\noindent
In particular, this would force $\dim_K \left( \chi_q^{\lambda/\mu} \right) = f^{\lambda/\mu}(q)$,
and would answer a question asked by Bj\"orner \cite[\S6, p. 207]{Bjorner}.
It suggests that perhaps there is a construction
of such a $\chi^{\lambda}_q$ in the spirit of the {\it cross-characteristic} 
$q$-analogue of Specht modules defined by James \cite{James}, which
also has dimension given by $f^{\lambda}(q)$.

\vskip .1 in
\noindent
{\sf Role 2.} 
Let $\lambda/\mu$ be an arbitrary skew shape.

\begin{question}
Does the hypothesized module $\chi^{\lambda/\mu}_q$ from Question~\ref{overarching-question} 
allow one to define the module
\begin{equation}
\label{module-via-q-Specht}
M:=\Hom_{\FF_q G}(\chi^{\lambda/\mu}_q,\FF_q[\xx])
\end{equation}
giving a definition for the $(q,t)$-fake degree $f^{\lambda/\mu}(q,t)$
as in part (d) of Section~\ref{qt-fake-degree-question}?
\end{question}


\vskip .1 in
\noindent
{\sf Role 3.}
Let $\lambda/\mu$ be an arbitary skew shape.
It follows from Schur-Weyl duality that one can 
re-interpret the usual Schur function principal specialization
as a Hilbert series in the following way:
\begin{equation}
\label{hilb-prin-spec}
s_{\lambda/\mu}(1,t,\ldots,t^{N-1}) := 
\Hilb\left( 
  \Hom_{\CC W}(\chi^{\lambda/\mu}, V^{\otimes n})
,t \right).
\end{equation}
Here $V=\CC^N$ is viewed as a graded vector space having
basis elements in degrees $(0,1,2,\ldots,N-1)$, inducing a grading
on the $n$-fold tensor space $V^{\otimes n}$, and 
$W=\Symm_n$ acts on $V^{\otimes n}$ by permuting tensor positions.

\begin{question}
Does the hypothesized $KG$-module $\chi^{\lambda/\mu}_q$ from 
Question~\ref{overarching-question}, playing the role of $\chi^{\lambda/\mu}$ in 
\eqref{hilb-prin-spec},  have a hypothesized accompanying
graded $KG$-module $V(N,n,q)$, playing the role of
$V^{\otimes n}$ in \eqref{hilb-prin-spec}, so that
$$
\MacSchurZ_{\lambda/\mu}(1,t,\ldots,t^{N-1}) := 
\Hilb\left( 
  \Hom_{K G}(\chi^{\lambda/\mu}_q, V(N,n,q))
,t \right) ?
$$
\end{question}

\section*{Acknowledgements}
The authors thank Peter Webb for helpful conversations, and thank
two anonymous referees for their suggestions.

\end{document}